\documentclass[11pt,reqno]{amsart}
\usepackage[a4paper,top=2.5cm, bottom=2.5cm,left=2.5cm, right=2.5cm,twoside]{geometry}

\usepackage{amsthm}
 \usepackage{amsmath}
 \usepackage{amsfonts}
 \usepackage{amssymb}
 \usepackage{amscd}
 \usepackage{url}
 \usepackage[T1]{fontenc}
 \usepackage[utf8x]{inputenc}
 \usepackage[english]{babel}
 \usepackage[mathscr]{eucal}
 \usepackage{lmodern}
 \usepackage{bera}
 \usepackage{float} 
 \usepackage{bold-extra}
 \usepackage{textcomp}
 \usepackage{graphicx}
 \usepackage{caption}
\usepackage{subcaption} 
\usepackage{multicol}
 \usepackage[dvipsnames,table]{xcolor}
 \usepackage{mathtools} 
 \usepackage{enumitem}
 \usepackage{tikz} 
 \usetikzlibrary{positioning, arrows.meta, decorations.markings, shapes.geometric, calc}
 \usepackage{tikz-cd}
 \usepackage{wasysym}
 \usepackage{todonotes}
 \usepackage{fancybox}
 \usepackage{fancyhdr}
 \usepackage{breqn}
 \usepackage[nodisplayskipstretch]{setspace}
\usepackage{hyperref}
\hypersetup{
	colorlinks=true,
	linkcolor=blue,
	filecolor=darkmagenta,      
	urlcolor=black,
	citecolor=red
} 

 \theoremstyle{definition}  
  \newtheorem{defn}{Definition}[section]
  \newtheorem{eg}[defn]{Example}

   \newtheorem{rmk}[defn]{Remark}

  \theoremstyle{plain}  
  \newtheorem{thm}[defn]{Theorem}
  \newtheorem{lem}[defn]{Lemma}
  \newtheorem{prop}[defn]{Proposition}
  \newtheorem{cor}[defn]{Corollary}

  \theoremstyle{remark}

  \newcommand{\la}{\langle}
  \newcommand{\ra}{\rangle}

 \numberwithin{equation}{section}
 \allowdisplaybreaks[4] 
 \setlist[enumerate]{font=\upshape,noitemsep, topsep=0pt} 
 \setlist[itemize]{noitemsep, topsep=0pt}

\makeatletter
\@namedef{subjclassname@2020}{\textup{2020} Mathematics Subject Classification}
\makeatother

\title[Eventually Entanglement Breaking Channels]{A complete characterization of  eventually entanglement breaking unital quantum channels}

\author{B. V. Rajarama Bhat}
\address{B. V. Rajarama Bhat, Statistics and Mathematics Unit, Indian Statistical Institute, Bangalore Centre, 8th Mile Mysore Road, RVCE Post, Bangalore 560059, India.}
\email{bvrajaramabhat@gmail.com, bhat@isibang.ac.in}

\author{Pankaj Dey}
\address{Pankaj Dey, Statistics and Mathematics Unit, Indian Statistical Institute, Bangalore Centre, 8th Mile Mysore Road, RVCE Post, Bangalore 560059, India.}
\email{pankajdey2022@gmail.com, pankajdey\_pd@isibang.ac.in}

\author{Biswarup Saha}
\address{Biswarup Saha, Statistics and Mathematics Unit, Indian Statistical Institute, Bangalore Centre, 8th Mile Mysore Road, RVCE Post, Bangalore 560059, India.}
\email{brsaha27700@gmail.com, rs\_math2201@isibang.ac.in}

\begin{document}


\begin{abstract}

We obtain a complete characterization of eventually entanglement breaking (EEB)  unital channels in terms of some simple commutation properties of their  eigenvectors. Specifically, the following properties are equivalent for a unital channel: (i) it is EEB; (ii) it is eventually positive partial transpose (PPT); (iii) it is eventually mixed twisted dephasing; and (iv) its eigenvectors corresponding to non-zero eigenvalues commute with all its peripheral eigenvectors. Similar results hold for one-parameter semigroups of unital quantum channels, where continuity forces such semigroups to become primitive.   We employ a matrix integral technique inspired by a result of Watrous. We also characterize the extreme points of the convex set of unital EB channels and, using this characterization and the theory of equiangular tight frames (ETFs), identify several extreme points beyond the twisted dephasing channels.
\end{abstract}

\date{\today}



\maketitle


\section{Introduction}

The study of quantum channels and their ability to preserve or destroy entanglement is central to quantum information theory. Among these, entanglement breaking channels form a particularly fundamental class, as they completely destroy entanglement for every input state. Such channels are of considerable theoretical and practical interest, as they represent a strong form of quantum noise. In particular, entanglement breaking channels have zero quantum capacity and, consequently, cannot be used for the reliable transmission of quantum information.

Let $\mathbb{C}^d$ denote the $d$-dimensional complex vector space and $\mathbb{M}_d$ denote the algebra of $d\times d$ complex matrices. Formally, a \textit{quantum channel} (or simply a channel) is defined as a completely positive (CP) and trace-preserving linear map between matrix algebras. A map on $\mathbb{M}_d$ of the form \[X\mapsto\sum\limits_{j=1}^r \langle v_j,Xv_j\rangle|u_j\rangle\langle u_j|,\]
where $\{u_j\}_{j=1}^r$ and $\{v_j\}_{j=1}^r$ are vectors in $\mathbb{C}^d$, is called an \textit{entanglement breaking} (EB) map.

 A channel is said to be \textit{eventually entanglement breaking} (EEB) if some finite iterate of the channel is entanglement breaking, and consequently all subsequent iterates are entanglement breaking. This notion captures the phenomenon that, although a single application of a noisy quantum process may not completely destroy entanglement, repeated applications may eventually do so.
A central motivation for studying eventually entanglement breaking channels arises from the PPT-squared conjecture, proposed by Christandl \cite{RusMB:JunM:DavK:HayP:AndW12, BCHW}. This conjecture asserts that the composition of any two PPT (positive partial transpose) maps is entanglement breaking. Although the conjecture remains open in full generality, significant progress has been made in this direction.

Rahaman, Jaques, and Paulsen \cite{RJP} proved that every unital PPT channel becomes entanglement breaking after finitely many iterations, that is, it is eventually entanglement breaking (EEB).  However, simple examples (Example \ref{EMD:eg1}) show that a channel need not be PPT for its iterates to become EEB. The class of such maps is much larger. In fact, it is dense in the space of all unital quantum channels. Here, we characterize the unital quantum channels that are EEB. Central to our study is a class of maps that we call as mixed twisted dephasing channels. Maps on $\mathbb{M}_d$ of the form
$$X\mapsto \sum _{j=1}^d \langle v_j, Xv_j\rangle |u_j\rangle \langle u_j|,$$
where $\{u_j\}_{j=1}^d$ and $\{v_j\}_{j=1}^d$ are orthonormal bases of $\mathbb{C}^d$, are known as \textit{twisted dephasing channels}. Clearly, they are EB. In fact, they are extreme points of the convex set of unital EB maps, although they do not cover all extreme points for $d\geq 3$. Their convex combinations form the \textit{mixed twisted dephasing} (MTD) channels. We observe that EEB channels are also eventually MTD. In fact, the classes of eventually EB, eventually MTD, eventually PPT, and eventually 2-copositive unital channels are all equal. Moreover, a unital quantum channel eventually possesses any of these  (and hence all of these) properties if and only if its eigenvectors corresponding to non-zero eigenvalues commute with all its peripheral eigenvectors. Thus, this property can be determined simply by examining the channel itself, without computing its iterates. This is our main result (Theorem \ref{EMD:thm5}) in the discrete-time situation and is described through the following diagram (Figure \ref{discrete_EEB_diagram}).

\begin{figure}[htbp]
\centering
\resizebox{\textwidth}{!}{
\begin{tikzpicture}[
    box/.style={
        rectangle, draw=black, thick, rounded corners, 
        inner sep=6pt, 
        align=center, font=\large 
    },
    charbox/.style={
        rectangle, draw=black, thick, rounded corners, inner sep=10pt, 
        align=center, font=\Large
    },
    implies/.style={
        double, double equal sign distance, -Implies, thick
    },
    iff/.style={
        double, double equal sign distance, Implies-Implies, thick
    },
    notimplies/.style={
        double, double equal sign distance, -Implies, thick,
        decoration={
            markings, 
            mark=at position 0.5 with {\draw[red, very thick, solid] (-4pt,-6pt) -- (4pt,6pt);}
        },
        postaction={decorate}
    },
    ce_text/.style={
        font=\normalsize, text=blue, align=center
    }
]

\node[box] (EB) {Entanglement\\breaking (EB)};
\node[box] (MD) [left=1.6cm of EB] {Mixed twisted\\dephasing\\(MTD)};
\node[box] (PPT) [right=1.6cm of EB] {Positive partial\\transpose\\(PPT)};
\node[box] (2COP) [right=1.6cm of PPT] {2-copositive};

\node[box] (EEB) [below=2.8cm of EB] {Eventually\\entanglement\\breaking (EEB)};
\node[box] (EMTD) at (MD |- EEB) {Eventually\\mixed twisted\\dephasing (EMTD)};
\node[box] (EPPT) at (PPT |- EEB) {Eventually\\positive partial\\transpose (EPPT)};
\node[box] (E2COP) at (2COP |- EEB) {Eventually\\2-copositive};

\node[box] (PRIM) [left=1.6cm of EMTD] {Primitive};

\path (EEB) -- (EPPT) node[midway] (MID) {};
\node[charbox] (CHAR) [below=2cm of MID] {
    All eigenvectors corresponding to  non-zero eigenvalues\\
    commute with all peripheral eigenvectors
};

\draw[implies] ([yshift=6pt]MD.east) -- ([yshift=6pt]EB.west);
\draw[notimplies] ([yshift=-6pt]EB.west) -- node[below, ce_text] {\ref{EMD:eg4},\\ \ref{EMD:cor4}} ([yshift=-6pt]MD.east); 

\draw[implies] ([yshift=6pt]EB.east) -- ([yshift=6pt]PPT.west);
\draw[notimplies] ([yshift=-6pt]PPT.west) -- node[below, ce_text] {\cite{HHH},\\ \cite{HO}} ([yshift=-6pt]EB.east);

\draw[implies] ([yshift=6pt]PPT.east) -- ([yshift=6pt]2COP.west);
\draw[notimplies] ([yshift=-6pt]2COP.west) -- node[below, ce_text] {\ref{EMD:eg3}} ([yshift=-6pt]PPT.east);

\draw[implies] ([yshift=6pt]PRIM.east) -- ([yshift=6pt]EMTD.west);
\draw[notimplies] ([yshift=-6pt]EMTD.west) -- node[below, ce_text] {\ref{EMD:eg2}} ([yshift=-6pt]PRIM.east);

\draw[iff] (EMTD) -- (EEB);
\draw[iff] (EEB) -- (EPPT);
\draw[iff] (EPPT) -- (E2COP);

\draw[implies] (MD) -- (EMTD);
\draw[implies] (EB) -- (EEB);
\draw[implies] (PPT) -- (EPPT);

\draw[implies] ([xshift=-6pt]2COP.south) -- ([xshift=-6pt]E2COP.north);
\draw[notimplies] ([xshift=6pt]E2COP.north) -- node[right, ce_text] {\ref{EMD:eg1}} ([xshift=6pt]2COP.south);

\draw[iff] (CHAR.160) -- (EMTD.south);
\draw[iff] (CHAR.115) -- (EEB.south);
\draw[iff] (CHAR.65) -- (EPPT.south);
\draw[iff] (CHAR.20) -- (E2COP.south);

\end{tikzpicture}
} 
\caption{EEB unital channels and their counterparts in discrete time}\label{discrete_EEB_diagram}
\end{figure}
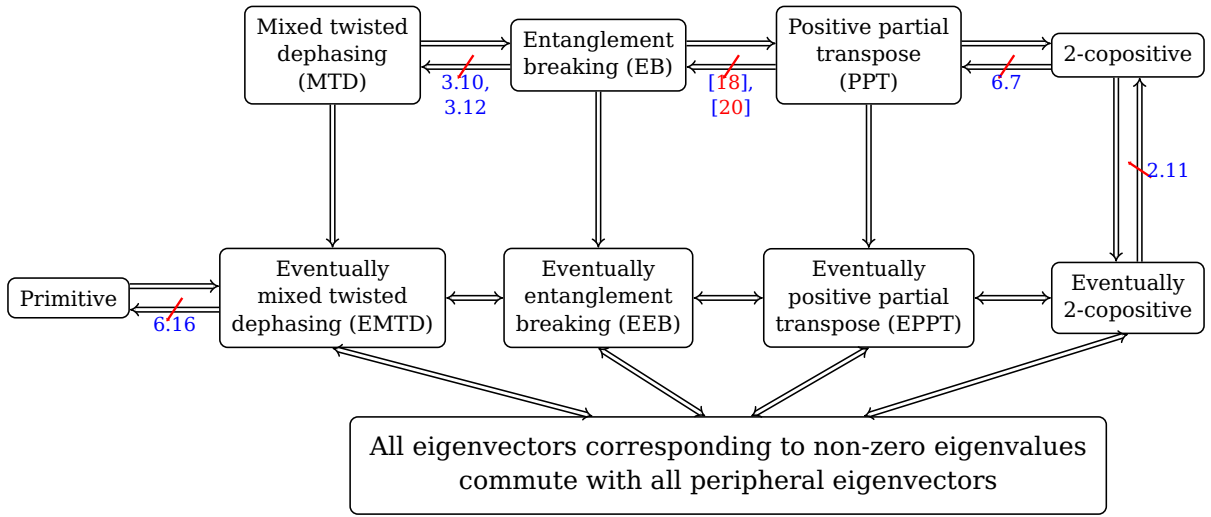

Briefly, our proof strategy is as follows: we begin by characterizing the extreme points of the compact convex set of unital EB channels on $\mathbb{M}_d$ and establishing the extremity of twisted dephasing channels. Building on our characterization, we construct several extreme points beyond this family using the theory of equiangular tight frames (ETFs). Then, motivated by the work of Watrous \cite{JW2} and Kribs et al. \cite{KLPR}, we derive two matrix integral formulae (Theorem \ref{EMD:thm3} and Theorem \ref{EMD:thm2}) for unital quantum channels. The first formula relates to the completely depolarizing channel $\delta_d$, which is the conditional expectation onto the $C^*$-algebra $\mathbb{C}I_d$ of scalar matrices in $\mathbb{M}_d$, and the integral is taken over all pairs of unitaries.  It is then generalized to conditional expectations onto an arbitrary commutative $C^*$-subalgebra $\mathcal{A}$ of $\mathbb{M}_d.$ Here, the integral is taken over pairs of unitaries from the commutant of $\mathcal{A}.$
With the help of these formulae, we are able to show that unital quantum channels close to these conditional expectation maps, provided they have some additional commutation properties, are actually MTD channels (Theorem \ref{EMD:cor2} and Theorem \ref{EMD:cor1}).  This serves as our main tool.

It is well-known that for any unital completely positive (CP) map, the linear span of its peripheral eigenvectors (that is, eigenvectors corresponding to eigenvalues of modulus 1) forms a $C^*$-subalgebra of $\mathbb{M}_d$ under a modified product called the extended Choi-Effros product. This $C^*$-algebra is called the peripheral Poisson boundary of the CP map. For unital quantum channels, this product matches with the original product. In other words, in this case, the peripheral Poisson boundary is a $C^*$-subalgebra of $\mathbb{M}_d$. If, furthermore, the channel is EB or EEB,  then this algebra is commutative. Asymptotically, the channel approaches the conditional expectation onto this commutative algebra, modulo an automorphism. Therefore, by applying the results derived from the matrix integral formulae, we conclude that these channels are mixed twisted dephasing. 

There is a subtle point to note in showing the eventual MTD property for a semigroup $\{\tau ^n\}_{n\geq 0}$. The composition of any CP map with an EB map (either from the left or the right) is again EB. Therefore, if $\tau ^{n_0}$ is EB, then $\tau ^n$ is EB for all $n\geq n_0.$ We cannot make a similar claim for MTD channels, because the fact that $\tau ^{n_0}$ is MTD does not immediately mean that $\tau ^{n_0+1}$ is MTD. However, Bhat and Devendra \cite{BD} showed that for any unital quantum channel $\tau $, the sequence $\{\tau ^n\}_{n\geq 0}$ is eventually mixed unitary. Because the composition of MTD channels with unitary conjugations are clearly MTD, the fact that $\tau^{n_0}$ is MTD does imply the eventual MTD property, although it remains unclear whether $\tau^{n_0+1}$ is MTD. 

We further investigate one-parameter semigroups of quantum channels. While Hanson, Rouz\'e, and Fran\c ca \cite{HRS} originally established that a continuous quantum Markov semigroup is primitive if and only if it is eventually EB, we strengthen their result for semigroups of unital quantum channels by extending the equivalence to include the eventual MTD, eventual PPT, and eventual $2$-copositive properties (Theorem \ref{EMD:thm4}). To establish this result, we adopt different approach, based on the interiority property of $\delta_d$ within the set of MTD channels (Theorem \ref{EMD:cor2}). Specifically, we characterize the primitivity of these semigroups in terms of their peripheral Poisson boundary (the space spanned by the eigenvectors of the generator corresponding to its purely imaginary eigenvalues), which is equivalent to the semigroup asymptotically approaching $\delta_d$. Thus, a primitive semigroup of quantum channels does not merely become EB after some time; it eventually enters the convex hull of twisted dephasing channels. A crucial distinction from the discrete setting is that, in continuous time, primitivity is equivalent to the eventual MTD property, whereas a discrete channel can be eventually MTD without being primitive (Example \ref{EMD:eg2}). A diagrammatic view of our main result characterizing eventually EB unital channels in continuous time is presented below (Figure \ref{continuous_EEB_diagram}).

\begin{figure}[htbp]
    \centering
    \begin{tikzpicture}[
        scale=0.85,
        transform shape,
        box/.style={
            rectangle, 
            draw=black, 
            thick, 
            rounded corners,
            minimum height=0.8cm, 
            minimum width=2.8cm,
            align=center, 
            fill=white, 
            font=\footnotesize
        },
        implies/.style={
            double, 
            double equal sign distance, 
            Implies-Implies, 
            thick
        }
    ]

        \node[box] (N1) at ({4.8*cos(90)}, {3.8*sin(90)}) {Eventually mixed\\ twisted dephasing (EMTD)};
        \node[box] (N2) at ({4.8*cos(38.57)}, {3.8*sin(38.57)}) {Eventually Entanglement\\ Breaking (EEB)};
        \node[box] (N3) at ({4.8*cos(-12.86)}, {3.8*sin(-12.86)}) {Eventually Positive\\ Partial Transpose (EPPT)};
        \node[box] (N4) at ({4.8*cos(-64.29)}, {3.8*sin(-64.29)}) {Eventually \\2-copositive};
        \node[box] (N5) at ({4.8*cos(-115.71)}, {3.8*sin(-115.71)}) {Trivial peripheral\\Poisson boundary};
        \node[box] (N6) at ({4.8*cos(-167.14)}, {3.8*sin(-167.14)}) {Primitive};
        \node[box] (N7) at ({4.8*cos(141.43)}, {3.8*sin(141.43)}) {Semigroup\\ approaches to $\delta_d$};

        \draw[implies] (N1) -- (N2);
        \draw[implies] (N2) -- (N3);
        \draw[implies] (N3) -- (N4);
        \draw[implies] (N4) -- (N5);
        \draw[implies] (N5) -- (N6);
        \draw[implies] (N6) -- (N7);
        \draw[implies] (N7) -- (N1);

    \end{tikzpicture}
    \caption{EEB unital channels and their counterparts in continuous time}\label{continuous_EEB_diagram}
    
\end{figure}
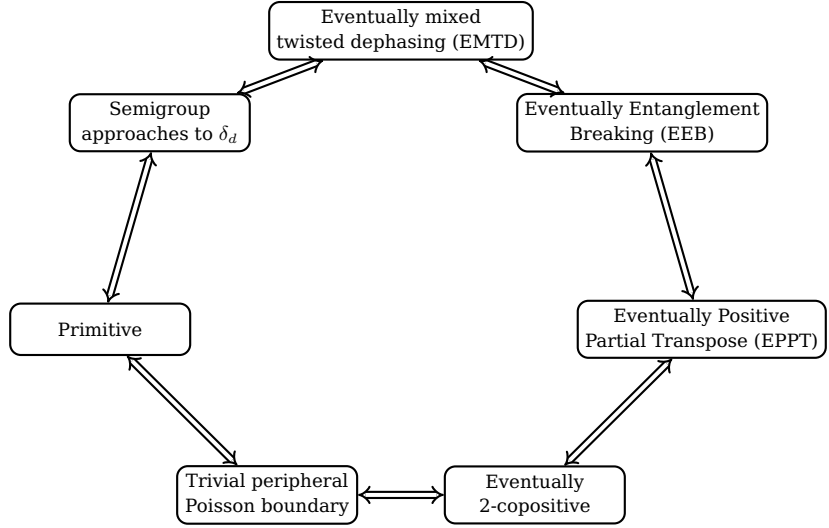

The structure of the paper is as follows. In Section \ref{priliminaries}, we provide the necessary preliminaries for our discussion. Section \ref{extreme point of UEBTP} is devoted to studying the extreme points of the compact convex set of unital EB channels.  We first characterize these points of the compact convex set of unital EB channels and show that twisted dephasing channels are indeed extreme points of this set (Theorem \ref{EMD:thm6} and Corollary \ref{EMD:cor3}). We utilize this characterization and the notion of equiangular tight frames to explicitly construct several extreme points beyond twisted dephasing channels (Example \ref{EMD:eg4} and Corollary \ref{EMD:cor4}).

In Sections \ref{completely depolarizing channel} and \ref{conditional expectation}, we develop the tools required for our main results by studying, respectively, the interiority of the completely depolarizing channel and that of the trace-preserving conditional expectation onto a commutative $C^*$-algebra within the set of MTD channels. Section \ref{discrete time} characterizes eventually EB unital channels in discrete time, while Section \ref{continuous time} does so in continuous time. We conclude by discussing an instructive class of examples in Section \ref{appendix}.

\section{Preliminaries}\label{priliminaries}

Let $\mathbb{C}^d$ denote the $d$-dimensional complex vector space equipped with the inner product defined as $\langle x,y\rangle=x^*y$ for $x,y\in\mathbb{C}^d$. We denote by $\{e_j\}_{j=1}^d$ the standard orthonormal basis of $\mathbb{C}^d$. Let $\mathbb{M}_{d_1,d_2}$ denote the space of all $d_1 \times d_2$ complex matrices, abbreviated as $\mathbb{M}_d$ when $d_1=d_2$. We denote by $\mathbb{M}_d^+$ the cone of all positive (semi-definite) matrices in $\mathbb{M}_d$. We begin with the following definition.

\begin{defn}(Separable matrix)
	A positive matrix $Z\in\mathbb{M}_{d_1}\otimes\mathbb{M}_{d_2}$ is said to be \textit{separable} if $Z=\sum_{j=1}^r X_j\otimes Y_j$ where $X_j\in\mathbb{M}_{d_1}^+$ and $Y_j\in\mathbb{M}_{d_2}^+$ for all $1\leq j\leq r$.
\end{defn}

We denote the identity map on $\mathbb{M}_d$ by $\operatorname{id}_d$. The transpose map $T:\mathbb{M}_d\rightarrow\mathbb{M}_d$ is defined by $T(X)=X^t$ for all $X\in\mathbb{M}_d$, where $X^t$ denotes the transpose of $X$ with respect to a fixed orthonormal basis. The trace of a matrix $X\in\mathbb{M}_d$ is denoted by $\text{Tr}(X)$.

\begin{defn}
	Let $\varphi:\mathbb{M}_{d_1}\rightarrow\mathbb{M}_{d_2}$ be a linear map. Then, $\varphi$ is said to be
	\begin{enumerate}
		\item \textit{positive} if $\varphi(\mathbb{M}_{d_1}^+)\subseteq\mathbb{M}_{d_2}^+$.
		\item \textit{$k$-positive} (respectively, \textit{$k$-copositive}) if $\varphi\otimes\operatorname{id}_k$ (respectively, $(\varphi\circ T)\otimes\operatorname{id}_k$) is positive.
		\item \textit{completely positive} (respectively, \textit{completely copositive}) if $\varphi$ is $k$-positive (respectively, $k$-copositive) for all $k\geq 1$.
		\item \textit{trace-preserving} (TP) if $\text{Tr}(\varphi(X))=\text{Tr}(X)$ for all $X\in\mathbb{M}_{d_1}$.
		\item \textit{quantum channel} if $\varphi$ is completely positive (CP) and TP.
		\item \textit{positive partial transpose} (PPT) if $\varphi$ is CP and completely copositive (coCP).
		\item \textit{entanglement breaking} (EB) if for every $k\geq 1$, the matrix $(\varphi\otimes\text{id}_k)(Z)$ is separable for every $Z\in(\mathbb{M}_{d_1}\otimes \mathbb{M}_k)^+$.
	\end{enumerate}
\end{defn}

 For vectors $x \in \mathbb{C}^{d_1}$ and $y \in \mathbb{C}^{d_2}$, we use the standard bra-ket notation $|x\rangle\langle y|$ to denote the rank-one operator from $\mathbb{C}^{d_2}$ to $\mathbb{C}^{d_1}$ defined by $z \mapsto \langle y, z\rangle x$. It is immediate to check that $A|x \rangle\langle y|B=|Ax \rangle\langle B^*y|$ for $A\in\mathbb{M}_{d_1}$ and $B\in\mathbb{M}_{d_2}$. We regard $\mathbb{M}_d$ as a Hilbert space equipped with the Hilbert-Schmidt inner product, that is, $\langle X, Y\rangle=\text{Tr}(X^*Y)$ for $X,Y\in\mathbb{M}_d$. Let $\varphi:\mathbb{M}_{d_1}\rightarrow\mathbb{M}_{d_2}$ be a linear map. The dual of $\varphi$ is defined as the unique linear map $\varphi^*:\mathbb{M}_{d_2}\rightarrow\mathbb{M}_{d_1}$ satisfying $\text{Tr}(X^*\varphi^*(Y))=\text{Tr}(\varphi(X)^*Y)$ for $X\in\mathbb{M}_{d_1}, Y\in\mathbb{M}_{d_2}$. It is straightforward to check that $\varphi$ is trace-preserving if and only if $\varphi^*(I_{d_2})=I_{d_1}$.

   In the realm of quantum information theory, the Choi-Jamio\l kowski isomorphism \cite{Cho75, Jam72}, commonly known as channel-state duality, illustrates the connection between quantum channels and quantum states. Let $\mathcal{B}({\mathbb{M}_{d_1},\mathbb{M}_{d_2}})$ denote the space of all linear maps from $\mathbb{M}_{d_1}$ to $\mathbb{M}_{d_2}$.

 \begin{thm}[Choi-Jamio\l kowski \cite{Cho75, Jam72}]\label{CJ isomorphism}
 	Let $\mathcal{J}$ be the linear map from $\mathcal{B}({\mathbb{M}_{d_1},\mathbb{M}_{d_2}})$ to $\mathbb{M}_{d_2}\otimes\mathbb{M}_{d_1}$ defined by 
 	$$\varphi\mapsto(\varphi\otimes\operatorname{id}_{d_1})|\xi_0\rangle\langle\xi_0|,$$
 	where $\xi_0=\sum_{j=1}^{d_1} e_j\otimes e_j$. Then, we have the following:
 	
 	\begin{enumerate}
 		\item $\mathcal{J}$ is a vector space isomorphism. 
 		\item $\varphi$ is positive map if and only if $\mathcal{J}(\varphi)$ is positive on simple tensor vectors, that is, $\langle x\otimes y, \mathcal{J}(\varphi)x\otimes y\rangle\geq 0$ for all $x\in\mathbb{C}^{d_1}$ and $y\in\mathbb{C}^{d_2}$. 
 		\item $\varphi$ is CP if and only if $\mathcal{J}(\varphi)$ is positive.
        \item $\varphi$ is EB if and only if $\mathcal{J}(\varphi)$ is separable.
        \item $\varphi$ is coCP if and only if $(\operatorname{id}_{d_2}\otimes T)(\mathcal{J}(\varphi))$ is positive.
 		\item $\varphi$ is unital if and only if $(\operatorname{id}_{d_2}\otimes\text{Tr})(\mathcal{J}(\varphi))=I_{d_2}$.
 		\item $\varphi$ is TP if and only if $(\text{Tr}\otimes \operatorname{id}_{d_1})(\mathcal{J}(\varphi))=I_{d_1}$.
 	\end{enumerate}
 \end{thm}

 The matrix $\mathcal{J}(\varphi)$ is known as the Choi-matrix of a linear map $\varphi$. Let $V\in\mathbb{M}_{d_1,d_2}$. The adjoint map $\operatorname{Ad}_V:\mathbb{M}_{d_1}\rightarrow\mathbb{M}_{d_2}$ is defined as $\operatorname{Ad}_V(X)=VXV^*$ for $X\in\mathbb{M}_{d_1}$. Check that the dual of $\operatorname{Ad}_V$ is $\text{Ad}_{V^*}$. Choi \cite{Cho75} and Kraus \cite{Kra71} characterized CP maps as follows.

 \begin{thm}[Choi-Kraus \cite{Cho75, Kra71}]
 	A linear map $\varphi:\mathbb{M}_{d_1}\rightarrow\mathbb{M}_{d_2}$ is completely positive if and only if there exist matrices $\{V_j\}_{j=1}^r\subseteq\mathbb{M}_{d_1,d_2}$ such that
 	\begin{align}\label{Choi-Kraus representation}
 		\varphi=\sum\limits_{j=1}^r \operatorname{Ad}_{V_j}.
 	\end{align}
 \end{thm}

 \begin{rmk}\label{metric operator space}
 	The representation given in Equation \ref{Choi-Kraus representation} is known as the Choi-Kraus representation of $\varphi$, and the operators $\{V_j\}_{j=1}^r$ are called the Kraus operators associated with $\varphi$. If $\varphi$ admits another representation with Kraus operators $\{W_j\}_{j=1}^k$, then it can be shown that $\text{span}\{W_j:1\leq j\leq k\}=\text{span}\{V_j:1\leq j\leq r\}$. We refer to $\text{span}\{V_j:1\leq j\leq r\}$ as the \textit{Kraus operator space} of $\varphi$ and denote it by $\mathscr{K}_\varphi$. The minimum number of Kraus operators required to represent a CP map $\varphi$ as in Equation \ref{Choi-Kraus representation} is called the Choi rank of $\varphi$, and is denoted by $\operatorname{CR}(\varphi)$. It can be checked that $\operatorname{CR}(\varphi)=r$ if and only if $\{V_j\}_{j=1}^r$ is linearly independent. Moreover, $\operatorname{CR}(\varphi)=r$ if and only if $\text{rank}(\mathcal{J}(\varphi))=r$.
  \end{rmk}

  The following theorem characterizes entanglement breaking maps.

  \begin{thm}[\cite{H, HSR}]\label{Holevo form}
  	Let $\varphi:\mathbb{M}_{d_1}\rightarrow\mathbb{M}_{d_2}$ be a linear map. Then, the following statements are equivalent:
  	\begin{enumerate}
  		\item $\varphi$ is entanglement breaking.
  		\item There exist rank-one matrices $\{V_j\}_{j=1}^r\subseteq\mathbb{M}_{d_1,d_2}$ such that $	\varphi=\sum_{j=1}^r \operatorname{Ad}_{V_j}$.
  		\item (Holevo form) There exists $F_j\in\mathbb{M}_{d_1}^+, R_j\in\mathbb{M}_{d_2}^+$ such that $\varphi(X)=\sum_{j=1}^k\text{Tr}(XF_j)R_j$.
  	\end{enumerate}
  \end{thm}

  \begin{rmk}
  	The minimum number of rank-one Kraus operators required to represent an EB map $\varphi$ as in Theorem \ref{Holevo form} (ii) is called the entanglement breaking rank of $\varphi$, and is denoted by $\operatorname{EBR}(\varphi)$ (see \cite{PPPR}). Clearly, $\operatorname{CR}(\varphi)\leq\operatorname{EBR}(\varphi)$.
  \end{rmk}

  Every EB map is PPT because if $\varphi=\sum_{j=1}^r\operatorname{Ad}_{|u_j\ra\la v_j|}$ is EB then $\varphi\circ T=\sum_{j=1}^r\operatorname{Ad}_{|u_j\ra\la \overline{v_j}|}$. The converse holds when $d_1d_2\leq 6$, but fails in general \cite{HHH, HO}. However, it was observed in \cite{BCHW} that the composition of any two of the known PPT maps is EB. This observation led the authors to formulate the PPT-squared conjecture: the composition of two PPT maps is always EB. While the general conjecture remains open, several results are known in this direction. In particular, Rahaman, Jaques, and Paulsen \cite{RJP} established the following result, for which we first recall the necessary definition.  
  
  \begin{defn}\label{defn of eventually EB}
  	A linear map $\varphi$ on $\mathbb{M}_d$ is said to be \textit{eventually entanglement breaking} (respectively, \textit{eventually PPT}, or \textit{eventually $2$-copositive}) if there exists $n_\circ\in\mathbb{N}$ such that $\varphi^n$ is entanglement breaking (respectively, PPT, or $2$-copositive) for all $n \ge n_\circ$. 
  \end{defn}

  \begin{rmk}
  	The properties of being EB, PPT, and $2$-copositive are preserved under composition with arbitrary CP maps. Consequently, a CP map $\varphi$ on $\mathbb{M}_d$ is eventually EB (respectively, eventually PPT or eventually $2$-copositive) if and only if there exists $n_\circ\in\mathbb{N}$ such that $\varphi^{n_0}$ is EB (respectively, PPT or $2$-copositive). Indeed, if $\varphi^{n_0}$ is EB (respectively, PPT or $2$-copositive) for some $n_\circ\in\mathbb{N}$, then $\varphi^n =  \varphi^{n-n_{0}}\circ\varphi^{n_0}$ is EB (respectively, PPT or $2$-copositive) for all $n \ge n_0$.
  \end{rmk}

  \begin{thm}[Rahaman-Jaques-Paulsen \cite{RJP}]\label{Rahaman's result}
  	Every unital PPT channel is eventually EB.
  \end{thm}

To establish this result, they employ the technique of the multiplicative domain. They also proved that the stabilized multiplicative domain of a unital quantum channel $\varphi$ is commutative if and only if some subsequence of $\{\varphi^{n}\}$ converges to an EB map (see also \cite{LG}). However, they demonstrated through counterexamples that the commutativity of the stabilized multiplicative domain alone does not guarantee that a channel is eventually EB. The class of eventually EB channels contains non-PPT channels, as demonstrated below.

The \textit{complete depolarizing channel} on $\mathbb{M}_d$, denoted by $\delta_d$, is defined as
 \begin{align}
 	\delta_d(X)=\frac{\text{Tr}(X)}{d}I_d.
 \end{align}
 Note that $\delta_d$ is the projection onto $\text{span}\{I_d\}$. Indeed, $\delta_d$ is EB, as it admits the following representation: $$\delta_d(X)=\frac{1}{d}\sum_{1\leq i,j\leq d}|e_i\rangle\langle e_j|X|e_j\rangle\langle e_i|.$$

\begin{eg}\label{EMD:eg1}
	Let $d\ge2$ be a natural number. Consider the unital quantum channel $\tau_\lambda=\lambda\text{id}_d+(1-\lambda)\delta_d$ for $\lambda\in (\frac{1}{d+1},1)$. Then, $\tau_\lambda$ is eventually EB. But $\tau_\lambda$ is not $2$-copositive, and hence not PPT. For details, we refer the reader to Example \ref{example in appendix}.
\end{eg}

Now, we construct a family of unital EB channels that will play a crucial role in our study. The \textit{completely dephasing channel} on $\mathbb{M}_d$, denoted by $\Delta_d$, is defined as
 \[\Delta_d(X)=\sum_{j=1}^d \langle e_j,X e_j\rangle|e_j\rangle\langle e_j|.\]
 Note that $\Delta_d$ is the projection onto the algebra of all diagonal matrices in $\mathbb{M}_d$. Let $\mathbb{U}(d)$ be the set of all unitary matrices in $\mathbb{M}_d$. For $U,V\in \mathbb{U}(d)$, we define
 \begin{align}\label{twisted dephasing channel}
 	  \beta_{U,V}=\text{Ad}_U\circ\Delta_d\circ\text{Ad}_{V^*}.
 \end{align}
Indeed, these are unital EB channels and we call them as $``$\textit{twisted dephasing channels}$"$. The following gives a Choi matrix description of twisted dephasing channels.

\begin{prop}\label{ChoiMatrixofTDchannels}
    Let $d\in\mathbb{N}$ and $\varphi$ be a linear map on $\mathbb{M}_d$. Then, the following statements are equivalent:
    \begin{enumerate}
        \item $\varphi$ is a twisted dephasing channel;
        \item There exist two orthonormal bases $\{u_k\}_{k=1}^d$ and $\{v_k\}_{k=1}^d$ of $\mathbb{C}^d$ such that the Choi matrix of $\varphi$ takes the form $\mathcal{J}(\varphi)=\sum_{k=1}^d |u_k\otimes v_k\ra\la u_k\otimes v_k|$.
        \item The Choi matrix $\mathcal{J}(\varphi)$ is separable and a projection satisfying the partial trace relations $(\text{id}_d\otimes \operatorname{Tr})(\mathcal{J}(\varphi))=I_d$ and $(\operatorname{Tr}\otimes\text{id}_d)(\mathcal{J}(\varphi))=I_d$.
    \end{enumerate}
\end{prop}
\begin{proof}
    The equivalence (1) $\Leftrightarrow$ (2) follows immediately from the canonical representation of twisted dephasing channels $\varphi = \sum_{k=1}^d \text{Ad}_{|Ue_k\ra\la Ve_k|}$ for $U,V\in\mathbb{U}(d)$. This yields $\mathcal{J}(\varphi) = \sum_{k=1}^d |Ue_k \otimes \overline{Ve_k}\ra\la Ue_k \otimes \overline{Ve_k}|$. The implication (2) $\Rightarrow$ (3) is evident from the given structure of the Choi matrix. We therefore proceed to prove (3) $\Rightarrow$ (2).
    
    Assume $\mathcal{J}(\varphi)$ is a separable projection satisfying the marginal trace conditions. By separability, $\mathcal{J}(\varphi) = \sum_j q_j |a_j \otimes b_j\ra\la a_j \otimes b_j|$ for some unit vectors $a_j, b_j \in \mathbb{C}^d$ and scalars $q_j > 0$. The partial trace constraints force
    \begin{equation}\label{EMD:eq3}
        \sum_j q_j |a_j\ra\la a_j| = I_d \quad \text{and} \quad \sum_j q_j |b_j\ra\la b_j| = I_d.
    \end{equation}
    Because $\mathcal{J}(\varphi)$ is a projection and $q_i |a_i \otimes b_i\ra\la a_i \otimes b_i| \le \mathcal{J}(\varphi)$, each $a_i \otimes b_i$ belongs to $\text{Range}(\mathcal{J}(\varphi))$. Thus, $\la a_i \otimes b_i, \mathcal{J}(\varphi)(a_i \otimes b_i)\ra = 1$, which expands to
    \begin{equation}\label{EMD:eq6}
        \sum_j q_j |\la a_i, a_j\ra|^2 |\la b_i, b_j\ra|^2 = 1.
    \end{equation}
    From Equation \ref{EMD:eq3}, evaluating the respective quadratic forms yields $\sum_j q_j |\la a_i, a_j\ra|^2 = 1$ and $\sum_j q_j |\la b_i, b_j\ra|^2 = 1$. Subtracting \eqref{EMD:eq6} from these two identities provides
    \[
        \sum_j q_j |\la a_i, a_j\ra|^2 (1 - |\la b_i, b_j\ra|^2) = 0 \quad \text{and} \quad \sum_j q_j |\la b_i, b_j\ra|^2 (1 - |\la a_i, a_j\ra|^2) = 0.
    \]
    Since $q_j > 0$ and the inner products are bounded by $1$, every term in both sums must vanish. Consequently, for any $i \neq j$, if $\la a_i, a_j\ra \neq 0$, then $|\la b_i, b_j\ra| = 1$, forcing $b_i$ and $b_j$ to generate the same rank-one projection, that is, $|b_i\ra\la b_i| = |b_j\ra\la b_j|$. By symmetry, this implies $|\la a_i, a_j\ra| = 1$ and hence $|a_i\ra\la a_i| = |a_j\ra\la a_j|$. Therefore, for $i \neq j$, either $\la a_i, a_j\ra = \la b_i, b_j\ra = 0$, or the rank-one projections $|a_i \otimes b_i\ra\la a_i \otimes b_i|$ and $|a_j \otimes b_j\ra\la a_j \otimes b_j|$ are equal.
    
    We can therefore merge the equal terms to rewrite $\mathcal{J}(\varphi) = \sum_{k=1}^r p_k |u_k \otimes v_k\ra\la u_k \otimes v_k|$, where $\{u_k\}_{k=1}^r$ and $\{v_k\}_{k=1}^r$ are now orthonormal sets in $\mathbb{C}^d$ and $p_k > 0$. The partial trace relations \eqref{EMD:eq3} now become $\sum_{k=1}^r p_k |u_k\ra\la u_k| = \sum_{k=1}^r p_k |v_k\ra\la v_k| = I_d$. The linear independence of $u_k$'s (or, $v_k$'s) forces $r = d$ and evaluating at some $u_l$ provides $p_l = 1$ for all $l$. This establishes that $\{u_k\}_{k=1}^d$ and $\{v_k\}_{k=1}^d$ are orthonormal bases, concluding (2).
\end{proof}
We now introduce the following definition.

 \begin{defn}[Mixed twisted dephasing channel]
 	A channel $\tau$ on $\mathbb{M}_d$ is said to be a \textit{mixed twisted dephasing} (MTD) channel if there exist $\{U_j, V_j: 1\leq j\leq n\}\subseteq \mathbb{U}(d)$ and a probability distribution $\{p_j\}_{j=1}^n$ (where $p_j \ge 0$ and $\sum_{j=1}^n p_j=1$) such that
 	\[\tau = \sum_{j=1}^n p_j \beta_{U_j, V_j}.\]
 \end{defn}

 The convex set of all mixed twisted dephasing channel on $\mathbb{M}_d$ is denoted by $\mathrm{MTD}(d)$. Due to the Carath\'{e}odory's theorem \cite{RocRT70}, this set is compact, as it is the convex hull of a compact set (the set of twisted dephasing channels) in a finite-dimensional space.  Moreover, it is closed under composition with mixed unitary channels from either side. Consequently, $\mathrm{MTD}(d)$ itself forms a semigroup under composition. Similarly, the set of all unital EB channels on $\mathbb{M}_d$ is denoted by $\mathrm{UEBTP}(d)$. Note that this set is also compact and convex. Indeed, $\mathrm{MTD}(d)\subseteq\mathrm{UEBTP}(d)$. The next proposition shows that this containment is, in fact, an equality when $d=2$.

 \begin{prop}\label{extreme point for M_2}
 	Every unital EB channel on $\mathbb{M}_2$ is a mixed twisted dephasing channel.
 \end{prop}

 \begin{proof}
 	Suppose $\tau \in \mathrm{UEBTP}(2)$ admits the representation $\tau = \sum_{k=1}^n p_k \text{Ad}_{|u_k\ra\la v_k|}$, where $\{u_k\}_{k=1}^n$ and $\{v_k\}_{k=1}^n$ are unit vectors in $\mathbb{C}^2$, and $p_k > 0$. Because $\tau$ is unital and trace-preserving, we have the identities
 	\begin{equation}\label{EMD:eq4}
 		\sum_{k=1}^n p_k |u_k\ra\la u_k| = I_2,\ \sum_{k=1}^n p_k |v_k\ra\la v_k| = I_2 \text{ and } \sum_{k=1}^np_k=2,
 	\end{equation}
 	where the last identity is obtained by taking the trace of either equation. For each $k$, let $u_k^\perp, v_k^\perp \in \mathbb{C}^2$ be two unit vectors orthogonal to $u_k$ and $v_k$ respectively. We then define the unitary matrices $U_k, V_k \in \mathbb{U}(2)$ via: $
 	U_k e_1 = u_k,\, U_k e_2 = u_k^\perp,\, V_k e_1 = v_k, \text{ and } V_k e_2 = v_k^\perp.$
 	We claim that $\tau$ can be expressed as the convex combination $\tau = \sum_{k=1}^n \frac{p_k}{2} \beta_{U_k, V_k}$. To verify this, we expand the action of $\beta_{U_k, V_k}$ on an arbitrary $X \in \mathbb{M}_2$. Utilizing the relations $I_2 = |u_k\ra\la u_k| + |u_k^\perp\ra\la u_k^\perp|$ and $I_2 = |v_k\ra\la v_k| + |v_k^\perp\ra\la v_k^\perp|$, we deduce
 	\begin{align*}
 		\beta_{U_k, V_k}(X)
 		&= \text{Tr}(X|v_k\ra\la v_k|)|u_k\ra\la u_k| + \text{Tr}(X|v_k^\perp\ra\la v_k^\perp|)|u_k^\perp\ra\la u_k^\perp| \\
 		&= \text{Tr}(X|v_k\ra\la v_k|)|u_k\ra\la u_k| + \text{Tr}\big(X(I_2 - |v_k\ra\la v_k|)\big)(I_2 - |u_k\ra\la u_k|) \\
 		&= 2\text{Tr}(X|v_k\ra\la v_k|)|u_k\ra\la u_k| + \text{Tr}(X)I_2 - \text{Tr}(X)|u_k\ra\la u_k| - \text{Tr}(X|v_k\ra\la v_k|)I_2.
 	\end{align*}
 	Summing over $k$ with weights $\frac{p_k}{2}$, the identities in \eqref{EMD:eq4} induce
 	\begin{align*}
 		&\sum_{k=1}^n \frac{p_k}{2} \beta_{U_k, V_k}(X)\\
 		&= \tau(X) + \frac{\text{Tr}(X)}{2} \bigg(\sum_{k=1}^n p_k\bigg)I_2 - \frac{\text{Tr}(X)}{2} \bigg(\sum_{k=1}^n p_k|u_k\ra\la u_k|\bigg) - \frac{1}{2} \text{Tr}\bigg(X\sum_{k=1}^n p_k|v_k\ra\la v_k|\bigg)I_2 \\
 		&= \tau(X) + \text{Tr}(X)I_2 - \frac{\text{Tr}(X)}{2}I_2 - \frac{\text{Tr}(X)}{2}I_2 \\
 		&= \tau(X).
 	\end{align*}
 	Therefore, $\tau$ is a mixed twisted dephasing channel. This completes the proof.
 \end{proof}

 To the best of our knowledge, the extreme points of the set $\mathrm{UEBTP}(d)$ have not yet been characterized. In the next section, we characterize the extreme points of the set $\mathrm{UEBTP}(d)$. In particular, we show that the twisted dephasing channels are extreme points of $\mathrm{UEBTP}(d)$. This implies, by Proposition \ref{extreme point for M_2}, that the extreme points of $\mathrm{UEBTP}(2)$ are precisely the twisted dephasing channels. However, for $d\geq 3$, there are extreme points of $\mathrm{UEBTP}(d)$ other than the twisted dephasing channels. For $d\geq 3$, we construct explicit examples of extreme points of $\mathrm{UEBTP}(d)$ other than the twisted dephasing channels, thereby showing that $\mathrm{MTD}(d)\subsetneq\mathrm{UEBTP}(d)$ for $d\geq 3$.

\section{Extreme points of the set of unital EB channels}\label{extreme point of UEBTP}

In this section, we characterize the extreme points of the set $\mathrm{UEBTP}(d)$. In particular, we show that the twisted dephasing channels are extreme points of $\mathrm{UEBTP}(d)$. But these do not exhaust all extreme points. We obtain several interesting examples of extreme points $\mathrm{UEBTP}(d)$ beyond twisted dephasing channels.

The set of all EB maps on $\mathbb{M}_d$ is denoted by $\mathrm{EB}(d)$. Let $\tau, \sigma\in\mathrm{EB}(d)$. We write $\sigma \preccurlyeq_{EB}\tau$ (read $``\sigma$ is EB-dominated by $\tau"$) if $\tau - \sigma\in\mathrm{EB}(d)$. The following theorem characterizes the extreme points of the set $\mathrm{UEBTP}(d)$.

\begin{thm}\label{EMD:thm6}
	Let $\tau\in\mathrm{UEBTP}(d)$ and define $F_0(\tau) = \{\text{Ad}_B : B \in \mathbb{M}_d, \, \text{rank}(B)=1, \, \text{Ad}_B \preccurlyeq_{EB} \tau\}$. Let $\{\text{Ad}_{B_k}\}_{k=1}^n \subseteq F_0(\tau)$ be a basis of $\mathcal{V}_{\tau}=\text{span}_{\mathbb{R}} F_0(\tau)$. Then the following are equivalent:
	\begin{enumerate}
		\item $\tau$ is an extreme point of $\mathrm{UEBTP}(d)$.
		\item The operators $\{B_k B_k^* \oplus B_k^* B_k\}_{k=1}^n \subseteq \mathbb{M}_d \oplus \mathbb{M}_d$ are linearly independent.
	\end{enumerate}
\end{thm}

\begin{proof}
	\noindent\textbf{(1) $\Rightarrow$ (2):} Let $\tau$ be an extreme point of $\mathrm{UEBTP}(d)$. Suppose
    $\sum_{k=1}^n c_k(B_k B_k^* \oplus B_k^* B_k) = 0$. Without loss of generality we can take  $c_k \in \mathbb{R}$. Let $\sigma = \sum_{k=1}^n c_k \text{Ad}_{B_k}$, decomposing as $\sigma = \sigma_+ - \sigma_-$ where $\sigma_{\pm} = \sum_{\pm c_k > 0} |c_k| \text{Ad}_{B_k}$ are EB maps. Choose $C > \max\{ \sum_{c_k \ge 0} c_k, \sum_{c_k < 0} |c_k| \}$. Since $\text{Ad}_{B_k} \preccurlyeq_{EB} \tau$, we obtain 
	\[
	\left(C - \sum_{c_k \ge 0} c_k\right)\tau \preccurlyeq_{EB} C\tau - \sigma_+.
	\]
	Thus, $C\tau - \sigma_+$ is EB, and analogously $C\tau - \sigma_-$ is EB. For sufficiently small $\epsilon > 0$ such that $\epsilon C \le 1$, the perturbed maps $\tau_{\pm} = \tau \pm \epsilon \sigma$ are EB, since 
	\[\tau_{\pm} = C^{-1}(C\tau - \sigma_{\mp}) + (C^{-1} - \epsilon)\sigma_{\mp} + \epsilon \sigma_{\pm}.\]
	The relation $\sigma(I_d) = \sigma^*(I_d) = 0$ guarantees $\tau_{\pm} \in \mathrm{UEBTP}(d)$. Extremity of $\tau = \frac{1}{2}(\tau_+ + \tau_-)$ forces $\tau_{\pm} = \tau$, yielding $\sigma = 0$. Then, by linear independence of $\{\text{Ad}_{B_k}\}_{k=1}^n$, we have $c_k = 0$ for all $k$. Therefore, $\{B_k B_k^* \oplus B_k^* B_k\}_{k=1}^n$ is linearly independent.

	 \noindent\textbf{(2) $\Rightarrow$ (1): } Let $\{B_k B_k^* \oplus B_k^* B_k\}_{k=1}^n$ be linearly independent. Suppose $\tau = p\sigma + (1-p)\sigma'$ with $\sigma, \sigma' \in \mathrm{UEBTP}(d)$ and $p \in (0,1)$. Then, $\sigma \preccurlyeq_{EB} p^{-1}\tau$, so $\sigma \in \mathcal{V}_\tau$. Expanding $\tau = \sum_{k=1}^n a_k \text{Ad}_{B_k}$ and $\sigma = \sum_{k=1}^n b_k \text{Ad}_{B_k}$ in the given basis with scalars $a_k,b_k\in\mathbb{R}$, the unital and trace-preserving conditions imply $\tau(I_d) - \sigma(I_d) = 0$ and $\tau^*(I_d) - \sigma^*(I_d) = 0$. Hence, $\sum_{k=1}^n (a_k - b_k)(B_k B_k^* \oplus B_k^* B_k) = 0$. Linear independence forces $a_k = b_k$, so $\sigma = \tau$. Therefore, $\tau$ is an extreme point of $\mathrm{UEBTP}(d)$.
\end{proof}

\begin{rmk}
	It should be noted that $\mathcal{V}_\tau$ in Theorem \ref{EMD:thm6} coincides with the known space $\text{span}_{\mathbb{R}} F(\tau)$ where $F(\tau) = \{ \sigma \in \mathrm{EB}(d) : \sigma \preccurlyeq_{EB} \lambda\tau \text{ for some } \lambda > 0 \}$ is the face of $\mathrm{EB}(d)$ containing $\tau$.
\end{rmk}

We will now discuss some consequences. Recall that $\mathscr{K}_\tau$ denotes the Kraus operator space of a CP map $\tau$ (see Remark \ref{metric operator space}).

\begin{cor}\label{extreme}
	Let $\tau \in \mathrm{UEBTP}(d)$ and define $\mathcal{W}_{\tau} = \text{span}_{\mathbb{R}}\{\text{Ad}_B : B \in \mathscr{K}_\tau, \, \text{rank}(B)=1\}$. If $\mathcal{W}_\tau$ admits a basis $\{\text{Ad}_{B_k}\}_{k=1}^m$, where   $\text{rank}(B_k)=1$ for all $k$, such that $\{B_k B_k^* \oplus B_k^* B_k\}_{k=1}^m$ is linearly independent, then $\tau$ is an extreme point of $\mathrm{UEBTP}(d)$.
\end{cor}

\begin{proof}
	    Let $\{\text{Ad}_{C_k}\}_{k=1}^n$ be a basis of $\mathcal{V}_\tau$ where $\text{Ad}_{C_k}\preccurlyeq_{EB}\tau$ and $\text{rank}(C_k)=1$ for all $1\leq k\leq n$. Because EB-dominance implies CP-dominance, the set $F_0(\tau)$ is contained in the spanning set of $\mathcal{W}_\tau$. Consequently, the set $\{\mathrm{Ad}_{C_k}\}_{k=1}^n$ of $\mathcal{V}_\tau = \text{span}_{\mathbb{R}}F_0(\tau)$ forms a linearly independent subset of $\mathcal{W}_\tau$. By Theorem \ref{EMD:thm6}, it suffices to prove that the corresponding direct sums $\{C_k C_k^* \oplus C_k^* C_k\}_{k=1}^n$ are linearly independent. Extend $\{\mathrm{Ad}_{C_k}\}_{k=1}^n$ to a basis $\{\mathrm{Ad}_{C_k}\}_{k=1}^m$ of $\mathcal{W}_\tau$. It follows that $\mathrm{Ad}_{C_k} = \sum_{l=1}^m a_{kl} \mathrm{Ad}_{B_l}$, where the coefficient matrix $[a_{kl}] \in \mathbb{M}_m(\mathbb{R})$ is invertible.  This gives $C_kC_k^*=\sum_{l=1}^m a_{kl}B_lB_l^*$ and $C_k^*C_k=\sum_{l=1}^m a_{kl}B_l^*B_l$. Now, we consider $\sum_{k=1}^mc_k(C_k C_k^* \oplus C_k^* C_k)=0$ where $c_k\in\mathbb{R}$ for all $k$. These imply that $\sum_{k,l=1}^mc_ka_{kl}(B_l B_l^* \oplus B_l^* B_l)=0$. The assumed linear independence of $\{B_l B_l^* \oplus B_l^* B_l\}_{l=1}^m$ forces $\sum_{k=1}^m c_k a_{kl} = 0$ for all $l$. Since $[a_{kl}]$ is invertible, this yields $c_k = 0$ for all $k$. Therefore, the subcollection $\{C_k C_k^* \oplus C_k^* C_k\}_{k=1}^n$ is linearly independent, confirming the extremity of $\tau$.
\end{proof}

\begin{prop}
	Every $\tau \in \mathrm{EB}(d)$ admits a decomposition $\tau = \sum_{k=1}^N \text{Ad}_{A_k}$, where each $A_k$ has rank one, and $\{\text{Ad}_{A_k}\}_{k=1}^N$ is linearly independent over $\mathbb{R}$.
\end{prop}

\begin{proof}
	We prove this via a standard Carath\'{e}odory-type reduction. Since $\tau \in \mathrm{EB}(d), \tau$ admits a representation $\tau = \sum_{k=1}^n p_k \text{Ad}_{|u_k\ra\la v_k|}$ for unit vectors $u_k, v_k \in \mathbb{C}^d$ and weights $p_k > 0$. If the set $\{\text{Ad}_{|u_k\ra\la v_k|}\}_{k=1}^n$ is linearly dependent over $\mathbb{R}$, there exist real scalars $c_k$, not all zero, satisfying $\sum_{k=1}^n c_k \text{Ad}_{|u_k\ra\la v_k|} = 0$. Evaluating this operator on $I_d$ and taking the trace forces $\sum_{k=1}^n c_k = 0$. Therefore, the set of coefficients must contain strictly negative values. Let $t = \min \left\{-\frac{p_k}{c_k} : c_k <0 \right\} > 0$, and let $k_0$ denote an index achieving this minimum. We can then rewrite the channel as $\tau = \sum_{k \neq k_0} (p_k + t c_k) \text{Ad}_{|u_k\ra\la v_k|}.$  By our choice of $t$, the purturbed weights $p_k + t c_k$ is non-negative and the $k_0$-th term vanishes. This reduces the number of conjugations by rank-one operators in the sum. Iterating this reduction process necessarily terminates in a linearly independent set.
\end{proof}

\begin{cor}\label{EMD:cor5}
	Let $\tau \in \mathrm{UEBTP}(d)$ be given by $\tau = \sum_{k=1}^N \text{Ad}_{A_k}$, where $\{A_k\}_{k=1}^N$ are rank-one matrices such that $\{\text{Ad}_{A_k}\}_{k=1}^N$ is linearly independent over $\mathbb{R}$. If $\tau$ is an extreme point of $\mathrm{UEBTP}(d)$, then $\{A_k A_k^* \oplus A_k^* A_k\}_{k=1}^N$ is linearly independent. The converse holds provided $\mathscr{K}_\tau$ contains no rank-one matrices other than scalar multiples of $\{A_k\}_{k=1}^N$. For any such extreme $\tau$, the value of $N$, and consequently the EB rank of $\tau$, is at most $2d^2-1$.
\end{cor}

\begin{proof}
The first part follows from Theorem \ref{EMD:thm6}, while the second part follows immediately from Corollary \ref{extreme}. Finally, whenever $\tau$ is extreme, the $N$ matrices $\{A_k A_k^* \oplus A_k^* A_k\}_{k=1}^N$ are linearly independent in the $(2d^2-1)$-dimensional space $\{X\oplus Y:\, X,Y\in\mathbb{M}_d,\text{Tr}(X)=\text{Tr}(Y)\}$, yielding $N\le 2d^2-1$.
\end{proof}

We next present an example of $\tau=\sum_{k=1}^4\text{Ad}_{A_k}\in\mathrm{UEBTP}(2)$ having Choi rank $4$ (hence, not extreme) but $\{A_kA_k^*\oplus A_k^*A_k\}_{k=1}^4$ is linearly independent. This demonstrates that the condition assumed in the converse part of Corollary \ref{EMD:cor5} cannot be dropped.

\begin{eg}\label{EMD:eg5}
	Let $\{e_1,e_2\}$ be the standard coordinate basis of $\mathbb{C}^2$ and $u=\frac{1}{\sqrt{2}}(e_1+e_2), v=\frac{1}{\sqrt{2}}(e_1-e_2)$ be two unit vectors. Define the rank-one matrices in $M_2$:
	\[A_1=\frac{1}{\sqrt{2}}|e_1\ra\la e_1|,\, A_2=\frac{1}{\sqrt{2}}|e_2\ra\la u|,\, A_3=\frac{1}{\sqrt{2}}|u\ra\la e_2|\text{ and }A_4=\frac{1}{\sqrt{2}}|v\ra\la v|.\]
	And, let $\tau=\sum_{k=1}^4\text{Ad}_{A_k}$. We can verify that these matrices are linearly independent and hence Choi rank of $\tau$ is $4$. It is easy to compute the following matrices
	\[A_1A_1^*=\frac{1}{2}\begin{bmatrix}
		1 & 0\\0& 0 
	\end{bmatrix},\, A_2A_2^*=\frac{1}{2}\begin{bmatrix}
		0 & 0\\0& 1 
	\end{bmatrix},\, A_3A_3^*=\frac{1}{4}\begin{bmatrix}
		1 & 1\\1& 1
	\end{bmatrix},\, A_4A_4^*=\frac{1}{4}\begin{bmatrix}
		\ \  1 & -1\\-1& \ \ 1
	\end{bmatrix},\]
	and $A_1^*A_1=A_1A_1^*,\, A_2^*A_2=A_3A_3^*,\, A_3^*A_3=A_2A_2^*,\, A_4^*A_4=A_4A_4^*$. Using these expressions of matrices, it is evident that $\sum_{k=1}^4A_kA_k^*=\sum_{k=1}^4 A_k^*A_k=I_2$, hence $\tau\in\mathrm{UEBTP}(2)$. Having Choi rank more than $2$, $\tau$ is not extreme point of $\mathrm{UEBTP}(2)$ (by Proposition \ref{extreme point for M_2}). Using the above matrix expressions, it is straightforward to verify that the matrices $\{A_kA_k^*\oplus A_k^*A_k\}_{k=1}^4$ are linearly independent. 
\end{eg}

Now, we show that the twisted dephasing channels are extreme points of $\mathrm{UEBTP}(d)$. In particular, by Proposition \ref{extreme point for M_2}, the extreme points of $\mathrm{UEBTP}(2)$ are precisely the twisted dephasing channels. However, for $d\geq 3$, there are extreme points of $\mathrm{UEBTP}(d)$ other than the twisted dephasing channels. We construct concrete examples of extreme points of $\mathrm{UEBTP}(d)$ whose Choi rank and EB rank coincide and can take any value between $d$ and $2d-2$. We need the following lemma.

\begin{lem}\label{EMD:lem10}
	Let $d\in\mathbb{N}$ and $X\in \mathbb{M}_d$. Suppose $X = \sum_{k=1}^s c_k |u_k\ra\la v_k|$ for non-zero scalars $c_k \in \mathbb{C}$ and vectors $u_k, v_k \in \mathbb{C}^d$. Let $r_u = \dim(\text{span}\{u_k\}_{k=1}^s)$ and $r_v = \dim(\text{span}\{v_k\}_{k=1}^s)$. Then, $s \ge r_u + r_v - \text{rank}(X)$.
\end{lem}

\begin{proof}
	Let $A\in \mathbb{M}_{d,s}$ be the matrix whose $k$-th column is the vector $c_ku_k$ and $B\in \mathbb{M}_{s,d}$ be the matrix whose $k$-th row is $v_k^*$.  The hypothesis reads $AB=X$ with $\text{rank}(A)=r_u$ and $\text{rank}(B)=r_v$. Then, by Sylvester's rank inequality, it follows that $\text{rank}(AB)\ge \text{rank}(A)+\text{rank}(B)-s$, hence $s\ge r_u+r_v-\text{rank}(X)$.
\end{proof}

\begin{cor}\label{EMD:cor3}
	Let $d, N \in \mathbb{N}$ with $d\leq N \le 2d-2$. Let $\{u_k\}_{k=1}^N$ and $\{v_k\}_{k=1}^N$ be unit vectors in $\mathbb{C}^d$ with weights $\{p_k\}_{k=1}^N \subseteq (0,\infty)$ satisfying:
	\begin{enumerate}
		\item $\sum_{k=1}^N p_k|u_k\ra\la u_k| = \sum_{k=1}^N p_k|v_k\ra\la v_k| = I_d$;
		\item Any $d$ vectors in $\{u_k\}_{k=1}^N$ and in $\{v_k\}_{k=1}^N$ are linearly independent;
		\item The set $\{|u_k\ra\la u_k| \oplus |v_k\ra\la v_k|\}_{k=1}^N$ is linearly independent.
	\end{enumerate}
	Then, the channel $\tau = \sum_{k=1}^N p_k\text{Ad}_{|u_k\ra\la v_k|}$ is an extreme point of $\mathrm{UEBTP}(d)$ with Choi rank and EB rank equal to $N$. In particular, twisted dephasing channels are extreme points of $\mathrm{UEBTP}(d)$.
\end{cor}

\begin{proof}
	By Corollary \ref{EMD:cor5}, it suffices to show that $\mathscr{K}_\tau$ contains no rank-one matrices beyond scalar multiples of $\{|u_k\ra\la v_k|\}_{k=1}^N$. Suppose $|x\ra\la y| = \sum_{k\in\mathbf{I}} c_k |u_k\ra\la v_k| \in \mathscr{K}_\tau ,$ with non-zero $c_k$ and support size $s = |\mathbf{I}|$. Condition (2) implies that $r_u = r_v = \min\{s, d\}$. Lemma \ref{EMD:lem10} then forces $s \ge 2\min\{s, d\} - 1$. If $s > d$, this implies $s \ge 2d - 1$, contradicting $s \le N \le 2d - 2$. Thus, $s \le d$, yielding $s \ge 2s - 1$, which forces $s=1$. This confirms that $\mathscr{K}_\tau$ admits no other rank-one operators other than scalar multiples of $\{|u_k\ra\la v_k|\}_{k=1}^N$.
	
	It remains to verify that its Choi rank and EB rank both equal to $N$. We will do this by showing $\{|u_k\ra\la v_k|\}_{k=1}^N$ is linearly independent. Suppose $\sum_{k\in\mathbf{I}}c_k|u_k\ra\la v_k|=0$ with non-zero $c_k$ and suppose size $s=|\mathbf{I}|$. Similar to the previous argument, using Lemma \ref{EMD:lem10}, we obtain $s\ge 2\min\{s,d\}$. It is not possible to have $s>d$, because if so, it will imply $s\ge 2d$ but we are given with $s\le N\le 2d-2$. Thus, we must have $s\le d$, hence $s\ge 2s\Rightarrow s=0$. This completes the proof.
\end{proof}

 \begin{cor}\label{EMD:cor7}
	The extreme points of $\mathrm{UEBTP}(2)$ are the twisted dephasing channels.
\end{cor}

\begin{proof}
	The proof is immediate by Corollary \ref{EMD:cor3} and Proposition \ref{extreme point for M_2}.
\end{proof}

\begin{eg}\label{EMD:eg4}
	Let $d \le N \le 2d-2$ and let $\zeta = e^{\frac{2\pi i}{N}}$. Define $w_k = \frac{1}{\sqrt{d}}(1, \zeta^k, \zeta^{2k}, \ldots, \zeta^{(d-1)k})^T \in \mathbb{C}^d$ for $1 \le k \le N$. Direct computation verifies the tight frame condition $\sum_{k=1}^N |w_k\ra\la w_k| = \frac{N}{d}I_d$. By the properties of Vandermonde matrices, any $d$ vectors from $\{w_k\}_{k=1}^N$ are linearly independent. Furthermore, because $d\leq N \le 2d-2 $, we have showed in the proof of Corollary \ref{EMD:cor3} that the operators $\{|w_k\ra\la w_k|\}_{k=1}^N$ are linearly independent. Therefore, by Corollary \ref{EMD:cor3}, the unital EB channel $\tau = \frac{d}{N}\sum_{k=1}^N \text{Ad}_{|w_k\ra\la w_k|}$ is extreme with Choi rank and EB rank equal to $N$. This produces a concrete examples of extreme points of $\mathrm{UEBTP}(d)$ whose Choi rank and EB rank coincide and can take any value between $d$ and $2d-2$.
\end{eg}

We conclude this section discussing another rich family of extreme points of $\mathrm{UEBTP}(d)$ whose Choi rank and EB rank coincide and can take any value between $d$ and $2d-2$ using the notion of equiangular tight frames (ETFs). The known extreme point of $\mathrm{UEBTP}(3)$ with Choi rank $4$ from \cite{HSR} emerges naturally as a special case of this framework. We recall the definition of the equiangular tight frames (ETFs).

\begin{defn}
	A set of unit vectors $\{w_k\}_{k=1}^N \subseteq \mathbb{C}^d$ (with $d \ge 1$) forms an \textit{equiangular tight frame} (ETF) if:
	\begin{enumerate}
		\item \textit{Equiangular}: $|\la w_k, w_l\ra| = \sqrt{\frac{N-d}{d(N-1)}}$ for all $1 \le k < l \le N$;
		\item \textit{Tight Frame}: $\sum_{k=1}^N |w_k\ra\la w_k| = \frac{N}{d}I_d$.
	\end{enumerate}
\end{defn}

It can be shown that if an ETF consisting of $N$ vectors in $\mathbb{C}^d$ exists, then $d\leq N\leq d^2$ (see, for instance, \cite[Theorem 5]{TroJA05}). ETFs of size $N=d$ (orthonormal bases) and $N=d+1$ (vertices of regular simplex centred at origin \cite{CK, SH}) exist in every dimension $d$. When $N=d^2$, the corresponding existence problem is known as the SIC POVM problem. SIC POVMs are known to exist in many dimensions, but their existence for every $d$ remains an open problem. In general, determining the pairs $(d,N)$ for which ETFs exist is a nontrivial and interesting problem. For further details on the existence of ETFs, we refer the reader to \cite{S, STDH}. For any set of unit vectors $\{w_k\}_{k=1}^N \subseteq \mathbb{C}^d$ spanning an $r$-dimensional subspace, the Welch bound \cite{W} states that
\begin{equation}\label{EMD:eq2}
	\max_{k \ne l} |\la w_k, w_l\ra| \ge \sqrt{\frac{N-r}{r(N-1)}}.
\end{equation}
Alternative proofs of this bound using Gram matrix are presented in \cite{SH, STDH}.

\begin{cor}\label{EMD:cor4}
	Let $\{w_k\}_{k=1}^N \subseteq\mathbb{C}^d$ be an ETF with $N \le 2d-2$. Then, the unital EB channel $\tau = \frac{d}{N}\sum_{k=1}^N \text{Ad}_{|w_k\ra\la w_k|}$ is an extreme point of $\mathrm{UEBTP}(d)$ with Choi and EB ranks equal to $N$.
\end{cor}

\begin{proof}
	We first show that $\mathscr{K}_\tau$ contains no rank-one matrices other than scalar multiples of $\{|w_k\ra\la w_k|\}_{k=1}^N$. Suppose $|x\ra\la y| = \sum_{k\in\mathbf{I}} c_k |w_k\ra\la w_k| \in \mathscr{K}_\tau$ with non-zero $c_k$ and support size $s = |\mathbf{I}| \ge 2$. Lemma \ref{EMD:lem10} implies $s \ge 2r - 1$, where $r = \dim(\text{span}\{w_k\}_{k\in\mathbf{I}})$. 
	Applying the Welch bound \eqref{EMD:eq2} to the subset $\{w_k\}_{k\in\mathbf{I}}$ yields
	\[
	\frac{N-d}{d(N-1)} = \max_{k \ne l} |\la w_k, w_l\ra|^2 \ge \frac{s-r}{r(s-1)} \ge \frac{1}{s+1} \ge \frac{1}{N+1}.
	\]
	This forces $N \ge 2d-1$, which contradicts $N \le 2d - 2$. Thus, $s=1$, establishing that $\mathscr{K}_\tau$ admits no other rank-one matrices. 
	
	Next, we show that $\{|w_k\ra\la w_k|\}_{k=1}^N$ is linearly independent. We consider the Gram matrix $G \in \mathbb{M}_N$ of the projectors, defined by $G_{kl} = \text{Tr}(|w_k\ra\la w_k||w_l\ra\la w_l|)$. The equiangular property yields $G = (1-c)I_N + c J_N$, where $c = \frac{N-d}{d(N-1)}$ and $J_N$ is the matrix whose entries are all equal to one. Because $c\in(0, 1)$, $G$ is strictly positive definite, ensuring $\{|w_k\ra\la w_k|\}_{k=1}^N$ are linearly independent. Therefore, by Corollary \ref{EMD:cor5}, $\tau$ is an extreme point of $\mathrm{UEBTP}(d)$ with Choi rank and EB rank equal to $N$.
\end{proof}

\begin{rmk}
    For $N>d\ge3$, the extreme channels constructed in Example \ref{EMD:eg4} and Corollary \ref{EMD:cor4} are not twisted dephasing channels. Because both families posses EB rank (as well as Choi rank) equal to $N$, whereas every twisted dephasing channel has those ranks to be exactly $d$. This explicitly demonstrates that $\mathrm{MTD}(d) \subsetneq \mathrm{UEBTP}(d)$ for all $d \ge 3$.  
\end{rmk}

\begin{rmk}
	Corollary \ref{EMD:cor4} shows that ETFs of size $N \le 2d-2$ give rise to extreme points of $\mathrm{UEBTP}(d)$ beyond the twisted dephasing channels. By Naimark's complement, an ETF of size $N$ in $\mathbb{C}^d$ exists if and only if a complementary ETF of size $N$ exists in $\mathbb{C}^{N-d}$. For details regarding this duality, reader may refer to \cite{FicM:MayBR21, STDH}.  Thus, constructing an ETF of size $N \le 2d-2$ in dimension $d$ is equivalent to finding one of size $N \ge 2d_0+2$ in the dual dimension $d_0 = N-d$. There are several interesting construction of such ETFs in the literature. We present one example from \cite[Section 2.1.2]{SH} showing infinite families of such configurations exist: for any prime $p \ge 3$ and $l \in \mathbb{N}$, set $d_0= p^l+1$ and $N = d_0^2-d_0+1$. Then, there exist integers $0 \le m_1 < \cdots < m_{d_0} < N$ such that the $d_0(d_0-1)$ differences $\{(m_{j}-m_{j'})\mod N:\, j\ne j'\}$ take all possible non-zero values $1,\ldots, N-1$. The harmonic vectors 
	\[
	v_k = \frac{1}{\sqrt{d_0}} \left(e^{\frac{2\pi i k m_1}{N}}, e^{\frac{2\pi i k m_2}{N}}, \ldots, e^{\frac{2\pi i k m_{d_0}}{N}}\right)^T \in \mathbb{C}^{d_0} \quad \text{for }k\in\{1,\ldots,N\},
	\]
	constitute an ETF of size $N$. As $d_0\ge 4$, it is easy to verify that this size $N=d_0^2-d_0+1\ge 2d_0+2$. Its Naimark complement consequently produces the desired ETF of size $N \le 2d-2$ in dimension $d = (d_0-1)^2$, guaranteeing a rich family of extreme points of $\mathrm{UEBTP}(d)$. Based on this idea of difference sets, more generalized family of such ETFs has been constructed in \cite{XZG}.
\end{rmk}

\section{Mixed twisted dephasing channels around the completely depolarizing channel}\label{completely depolarizing channel}

In this section, we show that there exists a neighbourhood of the completely depolarizing channel in which every unital quantum channel is mixed twisted dephasing. To establish this result, we employ a matrix integral approach inspired by a technique of Watrous \cite{JW2}.

Recall that $\mathbb{M}_d$ becomes a Hilbert space with respect to the Hilbert-Schmidt inner product: $\langle X, Y\rangle=\text{Tr}(X^*Y)$. Similarly, the space of superoperators, that is, linear maps on $\mathbb{M}_d$, is equipped with its own Hilbert–Schmidt inner product, making it a Hilbert space. In fact,  this inner product between superoperators happens to be just the Hilbert-Schmidt inner product between their respective Choi matrices. Consequently, we observe that the inner product between any two CP maps is non-negative, and in particular, real. We make extensive use of this inner product. Note the curious fact that for any unital map $\varphi$, $$\langle\varphi, \delta _d\rangle=1,$$
where $\delta_d$ is the completely depolarizing channel.

The following integral representation plays a central role  both in this section and the next. In view of the observation made above, the scalar coefficients of the integrand is a real valued function. The integration is with respect to the normalized Haar measure on the unitary group.

\begin{thm}\label{EMD:thm3}
	For any dimension $d$, every unital quantum channel $\tau$ on $\mathbb{M}_d$ admits a real-weighted integral representation over the family of the twisted dephasing channels $\beta_{U,V}$. Explicitly, $\tau$ can be expressed via the formula
	\begin{align}\label{EMD:eq1}
		\tau = \int_{U,V \in \mathbb{U}(d)} \Big( (d-1)(d+1)^2 \langle \beta_{U,V}, \tau \rangle - (d-1)(d+1)^2 + 1 \Big) \beta_{U,V} \, \mathrm{d}U \mathrm{d}V.
	\end{align}
\end{thm}

To prove this, we need a few lemmas.

  \begin{lem}[Watrous \cite{JW2}]\label{EMD:lem7}
	Let $d\geq 2$ and $\tau$ be a unital quantum channel on $\mathbb{M}_d$. Then, we have the following integral formula:
	\[\int_{U\in \mathbb{U}(d)}\langle\text{Ad}_U,\tau\ra\text{Ad}_U\ dU=\frac{1}{d^2-1}\tau+\frac{d^2-2}{d^2-1}\delta_d.\]
	Consequently, $\delta_d=\int_{U\in \mathbb{U}(d)}\text{Ad}_U\, \mathrm{d}U$.
\end{lem}

  \begin{lem}\label{EMD:lem2}
	For dimension $d\ge 2$, the twisted dephasing channels $\beta_{U,V}$ satisfy the following integral identities:
	\begin{enumerate}
		\item $\int_{U,V \in \mathbb{U}(d)} \beta_{U,V} \, \mathrm{d}U \mathrm{d}V = \delta_d$. Consequently, the completely depolarizing channel $\delta_d$ is  a mixed twisted dephasing channel.
		\item For any $U\in \mathbb{U}(d)$, 
		\[\int_{V\in \mathbb{U}(d)}\beta_{V,VU}\,\mathrm{d}V=\lambda(U)\operatorname{id}_d+(1-\lambda(U))\delta_d\;\text{ where }\lambda(U)=\frac{\text{Tr}(\Delta_d\circ\text{Ad}_{U^*})-1}{d^2-1}.\] In particular, $\int_{V \in \mathbb{U}(d)} \beta_{V,V} \, \mathrm{d}V = \frac{1}{d+1}\operatorname{id}_d + \frac{d}{d+1}\delta_d$.
	\end{enumerate}
\end{lem}

\begin{proof}
	Part (1) follows directly from the twirling identity $\int_{U \in \mathbb{U}(d)} \text{Ad}_U \, \mathrm{d}U = \delta_d$ and the fact that $\delta_d\circ\Delta_d\circ\delta_d=\delta_d$.
	
	To establish (2), let $\varphi_U = \int_{V \in \mathbb{U}(d)} \beta_{V,VU} \, \mathrm{d}V$. The translation invariance of the Haar measure ensures that $\varphi_U$ is unitarily covariant, that is, $\text{Ad}_W \circ \varphi_U \circ \text{Ad}_{W^*} = \varphi_U$ for all $W \in \mathbb{U}(d)$. By Schur's lemma, any such covariant map on $\mathbb{M}_d$ must be a linear combination of the identity map and the completely depolarizing channel. Because $\varphi_U$ is unital, this restricts to the affine combination $\varphi_U = \lambda(U)\, \text{id}_d + (1-\lambda(U))\,\delta_d$ for some scalar $\lambda(U) \in \mathbb{R}$. Now, we determine $\lambda(U)$ by evaluating the superoperator trace on both sides. The identity map on $\mathbb{M}_d$ has trace $d^2$, while $\delta_d$ acts as a rank-one projection onto the scalar matrices, yielding a trace of $1$. On the other hand, evaluating the trace directly from the integral definition gives $\text{Tr}(\varphi_U) =\int_{V\in \mathbb{U}(d)} \text{Tr}(\text{Ad}_V\circ\Delta_d\circ\text{Ad}_{U^*}\circ\text{Ad}_{V^*})\,\mathrm{d}V = \text{Tr}(\Delta_d\circ\text{Ad}_{U^*})$. Equating these traces yields $\lambda(U) = \frac{\text{Tr}(\Delta_d\circ\text{Ad}_{U^*})-1}{d^2-1}$, concluding the integral formula.

	In particular, if $U=I_d$, we have $\lambda(I_d)=\frac{\text{Tr}(\Delta_d)-1}{d^2-1}=\frac{1}{d+1}$ because $\Delta_d$ acts as the projection onto the $d$-dimensional diagonal algebra. Therefore, $\int_{V \in \mathbb{U}(d)} \beta_{V,V} \, \mathrm{d}V = \frac{1}{d+1}\text{id}_d + \frac{d}{d+1}\delta_d$.
\end{proof}

\begin{rmk}
	The depolarizing EB channel in part (2) of Lemma \ref{EMD:lem2} naturally appears in the context of SIC POVMs. Paulsen et al. \cite{PPPR} established that $\{w_k\}_{k=1}^{d^2}$ is a SIC POVM in $\mathbb{C}^d$ if and only if the EB rank of $\tau=\frac{1}{d}\sum_{k=1}^{d^2}\text{Ad}_{|w_k\ra\la w_k|}=\frac{1}{d+1}\operatorname{id}_d+\frac{d}{d+1}\delta_d$ is $d^2$. 
\end{rmk}

\begin{lem}\label{integral formula}
	Let $d\geq 2$ and $\tau$ be a unital quantum channel on $\mathbb{M}_d$. Then,  the following integral formula holds:
	\begin{align*}
		\int_{U,V \in \mathbb{U}(d)} \langle \beta_{U,V}, \tau \rangle \beta_{U,V} \, \mathrm{d}U \mathrm{d}V= \frac{1}{(d-1)(d+1)^2} \tau + \left( 1 - \frac{1}{(d-1)(d+1)^2} \right) \delta_d.
	\end{align*}
\end{lem}

\begin{proof}
		By the cyclic property of the Hilbert-Schmidt inner product, we can rewrite the scalar coefficient as $
		\langle \beta_{U,V}, \tau \rangle = \text{Tr}(\text{Ad}_V \circ \Delta_d \circ \text{Ad}_{U^*} \circ \tau) = \langle \text{Ad}_U, \tau \circ \text{Ad}_V \circ \Delta_d \rangle$. Substituting this into the integral and applying Lemma \ref{EMD:lem7} to the integration over $U$, we obtain
		\begin{align*}
			\int_{U,V \in \mathbb{U}(d)} &\langle \beta_{U,V}, \tau \rangle \beta_{U,V} \, \mathrm{d}U \mathrm{d}V \\
			&= \int_{V \in \mathbb{U}(d)} \left( \int_{U \in \mathbb{U}(d)} \langle \text{Ad}_U, \tau \circ \text{Ad}_V \circ \Delta_d \rangle \text{Ad}_U \, \mathrm{d}U \right) \circ \Delta_d \circ \text{Ad}_{V^*} \, \mathrm{d}V \\
			&= \int_{V \in \mathbb{U}(d)} \left( \frac{1}{d^2-1} \tau \circ \text{Ad}_V \circ \Delta_d + \frac{d^2-2}{d^2-1} \delta_d \right) \circ \Delta_d \circ \text{Ad}_{V^*} \, \mathrm{d}V \\
			&= \frac{1}{d^2-1} \tau \circ \left( \int_{V \in \mathbb{U}(d)} \beta_{V,V} \, \mathrm{d}V \right) + \frac{d^2-2}{d^2-1} \delta_d,
		\end{align*}
		where the final equality relies on the fact that the completely depolarizing channel $\delta_d$ is invariant under any unital, trace-preserving linear maps on $\mathbb{M}_d$. We now evaluate the remaining integral using part (2) of Lemma \ref{EMD:lem2}, which yields $\int_{V\in \mathbb{U}(d)}\beta_{V,V}\,\mathrm{d}V=\frac{1}{d+1}\operatorname{id}_d + \frac{d}{d+1}\delta_d$. Continuing the computation gives
		\begin{align*}
			\int_{U,V \in \mathbb{U}(d)} \langle \beta_{U,V}, \tau \rangle \beta_{U,V} \, \mathrm{d}U \mathrm{d}V &= \frac{1}{d^2-1} \tau \circ \left( \frac{1}{d+1} \operatorname{id}_d + \frac{d}{d+1} \delta_d \right) + \frac{d^2-2}{d^2-1} \delta_d \\
			&= \frac{1}{(d-1)(d+1)^2} \tau + \left( 1 - \frac{1}{(d-1)(d+1)^2} \right) \delta_d.
		\end{align*}
		This completes the proof.
\end{proof}

\begin{proof}[\textbf{Proof of Theorem \ref{EMD:thm3}}]
	  The proof follows from Lemma \ref{integral formula} by substituting the integral representation of $\delta_d$ given in part (1) of Lemma \ref{EMD:lem2} and then isolating $\tau$, which yields the desired identity.
\end{proof}

The following corollary provides a structural description of unital quantum channels.

  \begin{cor}\label{description of UCPTP}
	Every unital quantum channel $\tau$ on $\mathbb{M}_d$ can be written as an affine combination of the twisted dephasing channels $\beta_{U,V}$: 
	\[
	\tau = \sum_{j=1}^n a_j \beta_{U_j, V_j},
	\] where $\{a_j\}_{j=1}^n$ are real scalars and $\{U_j, V_j\}_{j=1}^n$ are unitary matrices in $\mathbb{M}_d$.
\end{cor}

\begin{proof}
	By Theorem \ref{EMD:thm3}, we write $\tau = \int_{U,V \in \mathbb{U}(d)} \rho(U, V) \beta_{U,V} \, \mathrm{d}U \mathrm{d}V$, where $\rho(U, V) = (d-1)(d+1)^2 \langle \beta_{U,V}, \tau \rangle - (d-1)(d+1)^2 + 1$. We decompose $\rho$ into positive and negative parts $\rho_\pm = \max\{\pm\rho, 0\}$, and subsequently define the maps $\tau_\pm = \int_{U,V \in \mathbb{U}(d)} \rho_\pm(U,V) \beta_{U,V} \, \mathrm{d}U \mathrm{d}V$ such that $\tau = \tau_+ - \tau_-$. Let $c_\pm = \int_{U,V \in \mathbb{U}(d)} \rho_\pm(U,V) \, \mathrm{d}U \mathrm{d}V$. Since the integral of $\rho$ is $1$, we have $c_+ - c_- = 1$. If $c_- = 0$, then $c_+ = 1$, placing $\tau$ in the closed convex hull of $\{\beta_{U,V}:\, U,V\in\mathbb{U}(d)\}$. By Carath\'{e}odory's theorem, $\tau$ is a finite convex combination of these channels. If $c_- > 0$, then $c_+ > 1$, and the normalized maps $c_\pm^{-1} \tau_\pm$ both reside in this closed convex hull. Applying Carath\'{e}odory's theorem again to each yields finite convex decompositions for $c_+^{-1} \tau_+$ and $c_-^{-1} \tau_-$. Combining via $\tau = c_+(c_+^{-1} \tau_+) - c_-(c_-^{-1} \tau_-)$ produces the desired finite real linear combination. Unitality of the involved maps confirms that it is an affine combination.
\end{proof}

We conclude this section with the following theorem, which plays a key role in characterizing eventually EB unital channels in continuous time.

\begin{thm}\label{EMD:cor2}
	For any dimension $d \in \mathbb{N}$, there exists an $\varepsilon > 0$ (depending only on $d$) such that every unital quantum channel $\tau$ on $\mathbb{M}_d$ satisfying $\|\tau - \delta_d\| < \varepsilon$ is a mixed twisted dephasing channel.
\end{thm}

\begin{proof}
	For a given unital quantum channel $\tau$ and unitaries $U, V \in \mathbb{U}(d)$, we define the continuous weight function $\rho(\tau, U, V) = (d-1)(d+1)^2 \langle \beta_{U,V}, \tau \rangle - (d-1)(d+1)^2 + 1$. By Theorem \ref{EMD:thm3}, $\tau$ can be expressed via the integral $\tau = \int_{U,V \in \mathbb{U}(d)} \rho(\tau, U, V) \beta_{U,V} \, \mathrm{d}U \mathrm{d}V$. We observe that $\rho(\delta_d, U, V) = 1$ identically across the unitary group. Now, we let $f(\tau) = \min_{U,V \in \mathbb{U}(d)} \rho(\tau, U, V)$ for unital quantum channel $\tau$. Because the unitary group $\mathbb{U}(d)$ is compact and $\rho$ is continuous, $f$ is well-defined and continuous with respect to $\tau$. Since $f(\delta_d) = 1$, continuity guarantees the existence of an $\varepsilon > 0$ such that $\|\tau - \delta_d\| < \varepsilon$ implies $f(\tau) > 0$. Therefore, for any channel $\tau$ within this $\varepsilon$-neighbourhood around $\delta_d$, the quantity $\rho(\tau, U, V)$ remains strictly positive for all $U,V$, yielding $\tau$ to be mixed twisted dephasing.
\end{proof}

\begin{rmk}
	A result of Gurvits and Barnum \cite[Corollary 3]{GB} states that the maximally mixed state $\frac{1}{d}I_{d^2}$ in $\mathbb{M}_{d^2}$ has a neighbourhood in which every normalized density matrix is separable. By Theorem \ref{CJ isomorphism}, this implies that there exists a neighbourhood of the completely depolarizing channel $\delta_d$ in which every unital quantum channel is EB. Theorem \ref{EMD:cor2} sharpens this result by showing that, in a neighbourhood of $\delta_d$, every unital quantum channel is not merely EB but, in fact, mixed twisted dephasing. In other words, by Proposition \ref{ChoiMatrixofTDchannels}, the maximally mixed state $\frac{1}{d}I_{d^2}$ in $\mathbb{M}_{d^2}$ has a neighbourhood in which every positive matrix with both partial traces equal to $I_d$ is not only separable but can also be expressed as a convex combination of separable projections. Also, note that our approach differs from that of Gurvits and Barnum.
\end{rmk}

\section{Mixed twisted dephasing channels around conditional expectations}\label{conditional expectation}

In this section, we study the interiority of the trace-preserving conditional expectation onto a commutative $C^*$-subalgebra of $\mathbb{M}_d$ within the set of mixed twisted dephasing channels. This serves as a key tool in our characterization of eventually EB unital channels. We begin by recalling the definition of a conditional expectation.

  \begin{defn}
	Let $\mathcal{A}$ be a unital $C^*$-subalgebra of $\mathbb{M}_d$. A \textit{conditional expectation} onto $\mathcal{A}$ is a CP map $\mathbb{E}:\mathbb{M}_d\rightarrow\mathcal{A}$ such that $\mathbb{E}(A)=A$ for all $A\in\mathcal{A}$.
 \end{defn}

Let $\mathcal{A}$ be a unital $C^*$-subalgebra of $\mathbb{M}_d$. It can be shown that $\mathbb{E}(AXB)=A\mathbb{E}(X)B$ for all $A,B\in\mathcal{A}$ and $X\in\mathbb{M}_d$ (see \cite{RP, HU}). Indeed, a conditional expectation onto $\mathcal{A}$ is not unique. However, the trace-preserving conditional expectation onto $\mathcal{A}$ is uniquely determined: it coincides with the orthogonal projection onto $\mathcal{A}$ (with respect to the Hilbert--Schmidt inner product). We denote this map by $\mathbb{E}_{\mathcal{A}}$. In particular, $\mathbb{E}_{\mathcal{A}}$ coincides with the completely depolarizing channel $\delta_d$ when $\mathcal{A}=\mathbb{C}I_d$, and with the completely dephasing channel $\Delta_d$ when $\mathcal{A}$ is the diagonal subalgebra of $\mathbb{M}_d$. Moreover, the trace-preserving conditional expectation $\mathbb{E}_{\mathcal{A}}$ admits the following integral representation:
\begin{align}\label{ce}
	\mathbb{E}_{\mathcal{A}}=\int_{U\in\mathbb{U}(\mathcal{A}')}\operatorname{Ad}_U dU,
\end{align}
where ${\mathbb{U}(\mathcal{A}')}=\mathbb{U}(d)\cap\mathcal{A}'$ is equipped with the normalized Haar measure. The following proposition shows that $\mathbb{E}_{\mathcal{A}}$ is mixed twisted dephasing channel when $\mathcal{A}$ is commutative.

  \begin{prop}\label{EMD:prop2}
	Let $\mathcal{A}$ be a $C^*$-subalgebra of $\mathbb{M}_d$. Then, \[\int_{U,V\in \mathbb{U}(\mathcal{A}')}\beta_{U,V}\ dU dV=\mathbb{E}_{\mathcal{A}}\circ\Delta_d\circ \mathbb{E}_{\mathcal{A}}.\]
	Consequently, if $\mathcal{A}$ consists of only diagonal matrices, then $\int_{U,V\in \mathbb{U}(\mathcal{A}')}\beta_{U,V}\ dU dV=\mathbb{E}_{\mathcal{A}}$. In particular, $\mathbb{E}_{\mathcal{A}}$ is mixed twisted dephasing whenever $\mathcal{A}$ is commutative.
\end{prop}

\begin{proof}
	The proof is immediate using Equation \ref{ce}.
\end{proof}

To study the interiority of the trace-preserving conditional expectation onto a commutative $C^*$-subalgebra of $\mathbb{M}_d$ within the set of mixed twisted dephasing channels, we need the following integral representation. Let $\tau$ be a unital quantum channel on $\mathbb{M}_d$. We denote by $\mathrm{Fix}(\tau)$ the fixed-point set of $\tau$. It can be shown that $\mathrm{Fix}(\tau)$ coincides with the commutant of the Kraus operator space $\mathscr{K}_\tau$, that is, $\mathrm{Fix}(\tau)=\mathscr{K}_\tau'$ (see \cite[Theorem 2.1]{DK}).

\begin{thm}\label{EMD:thm2}
	Let $\mathcal{A} = \bigoplus_{k=1}^r \mathbb{C}I_{m_k}$ be a commutative $C^*$-subalgebra of $\mathbb{M}_d$ where $d = \sum_{k=1}^r m_k$. Let $\Sigma_{\mathcal{A}}$ be the compact, convex set defined as \[\Sigma_{\mathcal{A}}=\{\tau:\, \tau\text{ is unital quantum channel on }\mathbb{M}_d\text{ with }\mathcal{A}\subseteq\text{Fix}(\tau)\cap \text{Range}(\tau)'\}.\]  Then, there exists a continuous weight function $\rho_{\mathcal{A}}: \Sigma_{\mathcal{A}} \times \mathbb{U}(\mathcal{A}') \times \mathbb{U}(\mathcal{A}') \to \mathbb{R}$ such that every $\tau \in \Sigma_{\mathcal{A}}$ admits the following integral representation:
	\[
	\tau = \int_{U,V \in \mathbb{U}(\mathcal{A}')} \rho_{\mathcal{A}}(\tau, U, V) \beta_{U,V} \, \mathrm{d}U\mathrm{d}V.
	\]
	Furthermore, this weight satisfies $\rho_{\mathcal{A}}(\mathbb{E}_{\mathcal{A}}, U, V) \equiv 1$ for all $U,V \in \mathbb{U}(\mathcal{A}')$.
\end{thm}

To prove this, we need a few lemmas.

\begin{lem}\label{EMD:lem5}
	Let $\tau$ be a unital quantum channel on $\mathbb{M}_d$, and let $\mathcal{A} \subseteq \text{Fix}(\tau)$ be a $C^*$-subalgebra with minimal central projections $P_1, \dots, P_r$. Then, every block subspace $P_k \mathbb{M}_d P_l$ is invariant under $\tau$. Consequently, the commutant of the center, $Z(\mathcal{A})'$, is a reducing subspace for $\tau$.
\end{lem}

\begin{proof}
	Let $X\in\mathbb{M}_d$. Since $\tau$ fixes the projections $P_k$, we have $\tau(P_k X P_l) = P_k \tau(X) P_l \in P_k \mathbb{M}_d P_l$. This invariance across all block components makes both $Z(\mathcal{A})' = \bigoplus_{k=1}^r P_k \mathbb{M}_d P_k$ and its orthogonal complement ${(Z(\mathcal{A})')}^\perp= \bigoplus_{k\neq l} P_k \mathbb{M}_d P_l$ invariant under $\tau$.
\end{proof}

Let us recall the following definition.

\begin{defn}[Direct sum]
	Let $m_1, \ldots, m_r$ be natural numbers summing to $d$, and let $\tau_k$ be a superoperator on $\mathbb{M}_{m_k}$ for $1 \le k \le r$. Their direct sum $\bigoplus_{k=1}^r \tau_k$ on $\mathbb{M}_d$ is defined as 
	\[\left(\bigoplus_{k=1}^r \tau_k\right)([X_{pq}]) = \bigoplus_{k=1}^r \tau_k(X_{kk}) \quad \text{for } X_{pq} \in \mathbb{M}_{m_p,m_q}.\] 
\end{defn}

\begin{lem}\label{EMD:lem1}
	Let $\mathcal{A} = \bigoplus_{k=1}^r \mathbb{C}I_{m_k}$ be a commutative $C^*$-subalgebra of $\mathbb{M}_d$ where $d = \sum_{k=1}^r m_k$. For any unital quantum channel $\tau$ on $\mathbb{M}_d$, the following are equivalent:
	\begin{enumerate}
		\item $\tau$ decomposes as $\tau = \bigoplus_{k=1}^r \tau_k$, where each $\tau_k$ is a unital quantum channel on $\mathbb{M}_{m_k}$;
		\item $\mathcal{A} \subseteq \text{Fix}(\tau) \cap \text{Range}(\tau)'$.
	\end{enumerate}
\end{lem}

\begin{proof}
	The implication (1) $\Rightarrow$ (2) is straightforward. To establish the converse, suppose $\mathcal{A} \subseteq \text{Fix}(\tau) \cap \text{Range}(\tau)'$. Lemma \ref{EMD:lem5} ensures that every block subspace $P_k \mathbb{M}_d P_l \cong \mathbb{M}_{m_k, m_l}$ is invariant under $\tau$. Furthermore, the hypothesis $\text{Range}(\tau) \subseteq \mathcal{A}' = \bigoplus_{k=1}^r \mathbb{M}_{m_k}$ forces that the image of $\tau$ is block-diagonal. Hence, $\tau$ must annihilate all off-diagonal blocks where $k \neq l$. Therefore, $\tau$ acts only on the diagonal subspaces individually, yielding the direct sum decomposition $\tau = \bigoplus_{k=1}^r \tau_k$. Finally, since $\tau$ is a unital quantum channel, it follows that each component $\tau_k$ is a unital quantum channel on its respective domain $\mathbb{M}_{m_k}$.
\end{proof}

    Let $m_1, \ldots, m_r$ be natural numbers summing to $d$. Given a linear map $\varphi$ on $\mathbb{M}_{m_l}$, we denote by $\iota_l(\varphi)$ its canonical zero-extension to $\mathbb{M}_d$. Explicitly, $\iota_l(\phi) = \bigoplus_{k=1}^r \tau_k$, where $\tau_l = \phi$ and $\tau_k \equiv 0$ for all $k \neq l$.

    \begin{lem}\label{EMD:lem4}
    	Let $\mathcal{A} = \bigoplus_{k=1}^r \mathbb{C}I_{m_k}$ be a commutative $C^*$-subalgebra of $\mathbb{M}_d$, where $d = \sum_{k=1}^r m_k$. Then, for any unital quantum channel $\tau$ on $\mathbb{M}_d$ satisfying $\mathcal{A} \subseteq \text{Fix}(\tau) \cap \text{Range}(\tau)'$, we have the following integral formula:
    	\small
    	\begin{align*}
    		\int_{U,V \in \mathbb{U}(\mathcal{A}')} \langle \beta_{U,V}, \tau \rangle \beta_{U,V} \, \mathrm{d}U\mathrm{d}V = \tau + (r-1)\mathbb{E}_{\mathcal{A}} + \sum_{m_k>1} \left( 1 - \frac{1}{(m_k-1)(m_k+1)^2} \right) (\iota_k(\delta_{m_k}) - \iota_k(\tau_k)),
    	\end{align*}
    	\normalsize
    	where $\mathbb{U}(\mathcal{A}') = \mathbb{U}(d) \cap \mathcal{A}' = \bigoplus_{k=1}^r \mathbb{U}(m_k)$.
    \end{lem}

    \begin{proof}
    	For $U, V \in \mathbb{U}(\mathcal{A}')$, the block structure of the commutant ensures $U = \bigoplus_{k=1}^r U_k$ and $V = \bigoplus_{k=1}^r V_k$, where $U_k, V_k \in \mathbb{U}(m_k)$. Consequently, the channel $\beta_{U,V}$ decomposes as the direct sum $\bigoplus_{k=1}^r \beta_{U_k,V_k}$. Furthermore, by Lemma \ref{EMD:lem1}, the condition on $\tau$ ensures that it similarly decomposes as $\tau = \bigoplus_{k=1}^r \tau_k$, with each $\tau_k$ being a unital quantum channel on $\mathbb{M}_{m_k}$. Thus, the integral splits block-wise:
    	\[
    	\int_{U,V \in \mathbb{U}(\mathcal{A}')} \langle \beta_{U,V}, \tau \rangle \beta_{U,V} \, \mathrm{d}U\mathrm{d}V = \bigoplus_{k=1}^r \int_{U,V \in \mathbb{U}(\mathcal{A}')} \langle \beta_{U,V}, \tau \rangle \beta_{U_k,V_k} \, \mathrm{d}U\mathrm{d}V.
    	\]
    	Now, since the inner product is additive across orthogonal blocks, we have $\langle \beta_{U,V}, \tau \rangle = \sum_{l=1}^r \langle \beta_{U_l,V_l}, \tau_l \rangle$. Therefore, we can evaluate the $k$-th summand by isolating the $l=k$ and $l \neq k$ terms. Utilizing the integral identities from Lemma \ref{EMD:lem2} and Lemma \ref{integral formula}, and noting that $\int \langle \beta_{U_l,V_l}, \tau_l \rangle \, \mathrm{d}U_l \mathrm{d}V_l = \langle \delta_{m_l}, \tau_l \rangle = 1$ due to trace preservation, we obtain for $m_k > 1$:
    	\begin{align*}
    		\int_{U,V \in \mathbb{U}(\mathcal{A}')}& \langle \beta_{U,V}, \tau \rangle \beta_{U_k,V_k} \, \mathrm{d}U \mathrm{d}V \\
    		&= \int_{U_k,V_k} \langle \beta_{U_k,V_k}, \tau_k \rangle \beta_{U_k,V_k} \, \mathrm{d}U_k \mathrm{d}V_k + \sum_{l \neq k} \int_{U,V \in \mathbb{U}(\mathcal{A}')} \langle \beta_{U_l,V_l}, \tau_l \rangle \beta_{U_k,V_k} \, \mathrm{d}U \mathrm{d}V \\
    		&= \left[ \frac{1}{(m_k-1)(m_k+1)^2} \tau_k + \left( 1 - \frac{1}{(m_k-1)(m_k+1)^2} \right) \delta_{m_k} \right] + \sum_{l \neq k} \langle \delta_{m_l}, \tau_l \rangle \delta_{m_k} \\
    		&= \tau_k + (r-1)\delta_{m_k} + \left( 1 - \frac{1}{(m_k-1)(m_k+1)^2} \right) (\delta_{m_k} - \tau_k).
    	\end{align*}
    	For the trivial case $m_k = 1$, the channels $\tau_k, \delta_{m_k}$ and $\beta_{U_k, V_k}$ coincide with the identity and the integral is exactly $\tau_k + (r-1)\delta_{m_k}$. Finally, taking the direct sum over all $k$ and identifying the conditional expectation $E_{\mathcal{A}} = \bigoplus_{k=1}^r \delta_{m_k}$ gives rise to the desired formula.
    \end{proof}

    \begin{lem}\label{EMD:lem8}
    	Let $\mathcal{A} = \bigoplus_{k=1}^r \mathbb{C}I_{m_k}$ be a commutative $C^*$-subalgebra of $\mathbb{M}_d$, where $d = \sum_{k=1}^r m_k$, and let $\varphi$ be a unital quantum channel on $\mathbb{M}_{m_k}$. Then,
    	\small
    	\[
    	\int_{U,V \in \mathbb{U}(\mathcal{A}')} \langle \beta_{U,V}, \iota_k(\varphi) \rangle \beta_{U,V} \, \mathrm{d}U\mathrm{d}V = 
    	\begin{cases}
    		\mathbb{E}_{\mathcal{A}} + \frac{1}{(m_k-1)(m_k+1)^2} \big( \iota_k(\varphi) - \iota_k(\delta_{m_k}) \big), & \text{if } m_k > 1; \\
    		\mathbb{E}_{\mathcal{A}}, & \text{if } m_k = 1.
    	\end{cases}
    	\]
    	\normalsize
    \end{lem}

    \begin{proof}
    	The computation is analogous to that in Lemma \ref{EMD:lem4}. Since $\iota_k(\varphi)$ is supported only on the $k$-th block, the inner product reduces to $\langle \beta_{U,V}, \iota_k(\varphi) \rangle = \langle \beta_{U_k,V_k}, \varphi \rangle$. Thus, for $m_k > 1$, factoring the integral over the orthogonal blocks yields
    	\begin{align*}
    		&\int_{U,V \in \mathbb{U}(\mathcal{A}')} \langle \beta_{U,V}, \iota_k(\varphi) \rangle \beta_{U,V} \, \mathrm{d}U\mathrm{d}V \\
    		&= \bigoplus_{l=1}^r \int_{U,V \in \mathbb{U}(\mathcal{A}')} \langle \beta_{U_k,V_k}, \varphi \rangle \beta_{U_l,V_l} \, \mathrm{d}U\mathrm{d}V \\
    		&= \left[ \frac{\varphi}{(m_k-1)(m_k+1)^2} + \left(1-\frac{1}{(m_k-1)(m_k+1)^2}\right)\delta_{m_k} \right] \bigoplus \left( \bigoplus_{l \neq k} \langle \delta_{m_k}, \varphi \rangle \delta_{m_l} \right) \\
    		&= \left[ \delta_{m_k} + \frac{1}{(m_k-1)(m_k+1)^2}(\varphi - \delta_{m_k}) \right] \bigoplus \left( \bigoplus_{l \neq k} \delta_{m_l} \right) \\
    		&= \mathbb{E}_{\mathcal{A}} + \frac{1}{(m_k-1)(m_k+1)^2} \big( \iota_k(\varphi) - \iota_k(\delta_{m_k}) \big).
    	\end{align*}
    	Finally, if $m_k = 1$, the integral over the $k$-th block trivially evaluates to $\delta_{m_k}$, hence the direct sum equals to $\mathbb{E}_{\mathcal{A}}$.
    \end{proof}

    \noindent\textit{\textbf{Proof of Theorem \ref{EMD:thm2}}.}
    	We define a function $\rho_{\mathcal{A}}: \Sigma_{\mathcal{A}} \times \mathbb{U}(\mathcal{A}') \times \mathbb{U}(\mathcal{A}') \to \mathbb{R}$ such that 
    	\[
    	\rho_{\mathcal{A}}(\tau, U, V) = \langle \beta_{U,V}, \tau + \sum_{m_k>1} c_k \iota_k(\tau_k) \rangle + C,
    	\]
    	 for all $\tau\in\Sigma_{\mathcal{A}}$ and $U,V\in\mathbb{U}(\mathcal{A}')$ where $\lambda_k = (m_k-1)(m_k+1)^2, c_k=\lambda_k-1$ and $C=-(r-1)-\sum_{m_k>1}c_k$. Indeed, $\rho_{\mathcal{A}}$ is continuous. By the linearity of the integral and inner product, integrating the weight function $\rho_{\mathcal{A}}(\tau, U, V)$ against $\beta_{U,V}$ yields:
    		\[\int\limits_{U,V\in\mathbb{U}(\mathcal{A}')} \langle \beta_{U,V}, \tau \rangle \beta_{U,V} \, \mathrm{d}U\mathrm{d}V + \sum_{m_k>1} c_k \int\limits_{U,V\in\mathbb{U}(\mathcal{A}')}  \langle \beta_{U,V}, \iota_k(\tau_k) \rangle \beta_{U,V} \, \mathrm{d}U\mathrm{d}V + C \int\limits_{U,V\in\mathbb{U}(\mathcal{A}')}  \beta_{U,V} \, \mathrm{d}U\mathrm{d}V.\]
    		Using the standard identity $\int_{U,V\in\mathbb{U}(\mathcal{A}')} \beta_{U,V} \, \mathrm{d}U\mathrm{d}V = \mathbb{E}_{\mathcal{A}}$ from Proposition \ref{EMD:prop2}, and substituting the explicit formulas from Lemmas \ref{EMD:lem4} and \ref{EMD:lem8}, we obtain the expression:
    		\small
    		\begin{align*}
    			&\tau + (r-1)\mathbb{E}_{\mathcal{A}} + \sum_{m_k>1} \left(1-\frac{1}{\lambda_k}\right)\big(\iota_k(\delta_{m_k}) - \iota_k(\tau_k)\big) + \sum_{m_k>1} c_k \left( E_{\mathcal{A}} + \frac{1}{\lambda_k} \big(\iota_k(\tau_k) - \iota_k(\delta_{m_k})\big) \right) + C \mathbb{E}_{\mathcal{A}}\\
    			&=\tau+\left(r-1+\sum_{m_k>1}c_k +C\right)\mathbb{E}_{\mathcal{A}}+\sum_{m_k>1}\left(1-\frac{1}{\lambda_k}-\frac{c_k}{\lambda_k}\right)\big(\iota_k(\delta_{m_k}) - \iota_k(\tau_k)\big).
    		\end{align*}
    		\normalsize
    		Now, substituting $c_k$ and $C$ into this expression establishes the integral representation of $\tau$. 
    		
    		Finally, for $\tau = \mathbb{E}_{\mathcal{A}}=\bigoplus_{k=1}^r \delta_{m_k}$, we have $\tau_k = \delta_{m_k}$ and $\langle \beta_{U,V}, \mathbb{E}_{\mathcal{A}} \rangle=\sum_{k=1}^r\la\beta_{U_k,V_k},\delta_{m_k}\ra=r$. Thus, for any $U,V\in \mathbb{U}(\mathcal{A}')$, evaluating the weight function yields:
    		\begin{align*}
    			\rho_{\mathcal{A}}(E_{\mathcal{A}}, U, V) = r + \sum_{m_k>1} c_k+ C =1.\tag*{\qed}
    		\end{align*}
    		
    		The following corollary provides a structural description.

    		 \begin{cor}
    			Let $\mathcal{A}$ be a commutative $C^*$-subalgebra of $\mathbb{M}_d$. Let $\tau$ be a unital quantum channel on $\mathbb{M}_d$ satisfying $\mathcal{A}\subseteq\text{Fix}(\tau)\cap\text{Range}(\tau)'$. Then, there exist unitaries $\{U_i, V_i\}_{i=1}^n\subseteq\mathbb{U}(\mathcal{A}')$ and real scalars $\{a_i\}_{i=1}^n$ satisfying $\sum_{i=1}^n a_i = 1$ such that 
    			\[
    			\tau =\sum_{i=1}^n a_i \beta_{U_i,V_i}.
    			\]
    		\end{cor}

    		\begin{proof}
    			Since $\mathrm{MTD}(d)$ is closed under composition with unitary channels from either side, without loss of generality, we can take $\mathcal{A} = \bigoplus_{k=1}^r \mathbb{C}I_{m_k}$ where $d = \sum_{k=1}^r m_k$.
                Then, the result follows with the help of Theorem \ref{EMD:thm2} and Carath\'{e}odory's Theorem, analogous to the proof of Corollary~\ref{description of UCPTP}.
    		\end{proof}

    		The following theorem is a key tool in characterizing eventually EB unital quantum channels in the next section.
    		
    		\begin{thm}\label{EMD:cor1}
    			Let $\mathcal{A}$ be a commutative $C^*$-subalgebra of $\mathbb{M}_d$. Then, there exists $\varepsilon>0$ such that every unital quantum channel $\tau$, satisfying the properties $\mathcal{A}\subseteq\text{Fix}(\tau)\cap\text{Range}(\tau)'$ and $\|\mathbb{E}_{\mathcal{A}}-\tau\|<\varepsilon$, is mixed twisted dephasing.
    		\end{thm}
    		
    		\begin{proof}
    			Without loss of generality, we can take $\mathcal{A} = \bigoplus_{k=1}^r \mathbb{C}I_{m_k}$ where $d = \sum_{k=1}^r m_k$. The proof follows from Theorem~\ref{EMD:thm2} and is analogous to that of Theorem~\ref{EMD:cor2}.
    		\end{proof}

   The condition that $\tau$ fixes $\mathcal{A}$ in Theorem \ref{EMD:cor1} alone is not sufficient to guarantee the stated conclusion, as demonstrated by the following example.

  \begin{eg}\label{example in a nbd of expectation}
  	Let $d\ge2$ and $\mathcal{A}$ denote the diagonal subalgebra of $\mathbb{M}_d$. Then, $\mathbb{E}_{\mathcal{A}}=\Delta_d$. Consider $\tau=\lambda \operatorname{id}_d+(1-\lambda)\Delta_d$ where $\lambda\in(0,1)$. Indeed, $\tau$ is unital quantum channel that fixes $\mathcal{A}$ for all $\lambda$. But $\tau$ is not $2$-copositive (hence not mixed twisted dephasing) for any $\lambda$ because the following matrix \[(\tau\circ T)^{(2)}\left(\begin{bmatrix}
  		E_{11} & E_{12}\\ E_{21} & E_{22}
  	\end{bmatrix}\right)=\begin{bmatrix}
  		E_{11} & \lambda E_{21}\\ \lambda E_{12} & E_{22}
  	\end{bmatrix},\]
  	has eigenvalue $-\lambda<0$. Note that $\text{Range}(\tau)$ is not contained in $\mathcal{A}'$ for any $\lambda$. The shows that the condition that $\tau$ fixes $\mathcal{A}$ in Theorem \ref{EMD:cor1} alone is not sufficient to guarantee the conclusion. 
  \end{eg}

  A natural question that arises at this point is the following: given a commutative $C^*$-subalgebra $\mathcal{A}$ of $\mathbb{M}_d$, does there exist a neighbourhood of $\mathbb{E}_{\mathcal{A}}$ in which every unital $2$-copositive quantum channel that fixes $\mathcal{A}$ is mixed twisted dephasing? We answer this question in the affirmative below. The following theorem states that given a commutative $C^*$-algebra $\mathcal{A}$ of $\mathbb{M}_d$, the set of all unital quantum channel satisfying $\mathcal{A}\subseteq\text{Fix}(\tau)\cap\text{Range}(\tau)'$ contains the set of all unital $2$-copositive channel that fixes $\mathcal{A}$ pointwise.

 \begin{thm}\label{EMD:prop1}
 	Let $\mathcal{A}$ be a commutative $C^*$-subalgebra of $\mathbb{M}_d$. Then, every unital $2$-copositive channel on $\mathbb{M}_d$ that fixes $\mathcal{A}$ pointwise has range contained in $\mathcal{A}'$.
 \end{thm}

 \begin{proof}
 	Without loss of generality, we can take $ \mathcal{A} = \bigoplus_{k=1}^r \mathbb{C}I_{m_k}$ where $d = \sum_{k=1}^r m_k$. By Lemma \ref{EMD:lem5}, every block subspace $P_k \mathbb{M}_d P_l \cong \mathbb{M}_{m_k,m_l}$ is invariant under $\tau$. We claim that $\tau$ vanishes entirely on the off-diagonal blocks $\mathbb{M}_{m_k,m_l}$ for $k \neq l$. Fix distinct indices $k, l$, and choose non-zero vectors $u, u'\in \mathbb{C}^{m_k}$ and $v,v'\in \mathbb{C}^{m_l}$. Let $\xi=\overline{u}\otimes e_1+\overline{v}\otimes e_2\in\mathbb{C}^d\otimes \mathbb{C}^2$ (where $\{e_1,e_2\}$ is a standard orthonormal basis for $\mathbb{C}^2$). Then, we evaluate $(\tau\circ T)\otimes\text{id}_{2}$ on the rank-$1$ positive matrix $|\xi\ra\la\xi|\in \mathbb{M}_d\otimes \mathbb{M}_2$ and compute its off-diagonal entry:
 	\[
 	\la u'\otimes e_2,((\tau\circ T)\otimes\text{id}_{2})(|\xi\ra\la\xi|)(v'\otimes e_1)\ra=\la u',\tau(|u\ra\la v|)v'\ra.
 	\]
 	Similarly, computing the associated diagonal entry yields
 	\[
 	\langle v'\otimes e_1, ((\tau\circ T)\otimes\text{id}_{2})(|\xi\ra\la\xi|)(v'\otimes e_1) \rangle = \langle v', \tau(|u\rangle\langle u|) v' \rangle = \text{Tr}\big(\tau(|u\rangle\langle u|)|v'\rangle\langle v'|\big).
 	\]
 	Since $|u\rangle\langle u| \in \mathbb{M}_{m_k}$, its image under $\tau$ belongs to $\mathbb{M}_{m_k}$ due to the invariance. Whereas $|v'\rangle\langle v'| \in \mathbb{M}_{m_l}$ lies in the orthogonal complement of $\mathbb{M}_{m_k}$. Therefore, the trace above evaluates to zero.

 	Because $\tau$ is a $2$-copositive channel, the map $\tau \circ T$ is $2$-positive. This implies that $((\tau\circ T)\otimes\text{id}_{2})(|\xi\ra\la\xi|) \ge 0$. A positive semi-definite matrix with a zero on its diagonal must have identically zero entries across that entire row and column. Thus, $\la u'\otimes e_2,((\tau\circ T)\otimes\text{id}_{2})(|\xi\ra\la\xi|)(v'\otimes e_1)\ra = 0$, proving that $\tau(|u\rangle\langle v|) = 0$. Therefore, $\tau$ vanishes completely on the off-diagonal block $\mathbb{M}_{m_k,m_l}$. Hence, $\tau$ acts non-trivially only on the diagonal blocks, showing the range of $\tau$ is contained inside the block diagonal matrices $\mathcal{A}'=\bigoplus_{k=1}^r \mathbb{M}_{m_k}$.
 \end{proof}

 \begin{cor}
 	Let $\mathcal{A}$ be a commutative $C^*$-subalgebra of $\mathbb{M}_d$. Then, there exists $\varepsilon>0$ such that every unital $2$-copositive quantum channel $\tau$ that fixes $\mathcal{A}$ satisfying $\|\mathbb{E}_{\mathcal{A}}-\tau\|<\varepsilon$ is mixed twisted dephasing.
 \end{cor}

 \begin{proof}
 	The proof follows by Theorem \ref{EMD:prop1} and Theorem \ref{EMD:cor1}.
 \end{proof}

  \section{Eventually EB unital channels in discrete time}\label{discrete time}

In this section, we characterize eventually EB unital channels in terms of some commutation relations among their eigenvectors. Consequently, we establish that for a unital quantum channel, the notions of being eventually EB, eventually PPT, eventually $2$-copositive, and eventually MTD are equivalent.  We begin with the following definition.

\begin{defn}
	Let $\tau:\mathbb{M}_d\rightarrow\mathbb{M}_d$ be a linear map. Then
	\begin{enumerate}
		\item The \textit{multiplicative domain} of $\tau$, denoted by $\mathscr{M}_\tau$, is defined as
		\begin{align*}
			\mathscr{M}_\tau:=\{X\in\mathbb{M}_d: \tau(XY)=\tau(X)\tau(Y), \tau(YX)=\tau(Y)\tau(X) \text{ for all }Y\in\mathbb{M}_d\}.
		\end{align*}
		\item The \textit{peripheral spectrum} of $\tau$ is the set of all eigenvalues of $\tau$ having modulus $1$, that is, $\{\lambda\in\text{Spec}(\tau):\vert\lambda\vert=1\}$. Such eigenvalues are called \textit{peripheral eigenvalues} of $\tau$, and the corresponding eigenvectors are called \textit{peripheral eigenvectors} of $\tau$.
		\item The \textit{peripheral space} of $\tau$, denoted by $\mathcal{P}(\tau)$, is defined to be the subspace spanned by the \textit{peripheral eigenvectors}:
		\begin{align*}
			\mathcal{P}(\tau):=\text{span}\{X \in \mathbb{M}_d : \tau(X) = \lambda X \text{ for some } |\lambda| = 1\}.
		\end{align*}
	\end{enumerate}
\end{defn}

Let $\tau$ be a unital quantum channel on $\mathbb{M}_d$. Note that the peripheral space $\mathcal{P}(\tau)$ is unital and self-adjoint preserving. By the Kadison-Schwarz inequality, it follows that the peripheral space is contained in the multiplicative domain, that is, $\mathcal{P}(\tau)\subseteq \mathscr{M}_\tau$. Bhat et al. \cite{BKT, BKT2} proved that $\mathcal{P}(\tau)$ is a $C^*$-subalgebra of $\mathbb{M}_d$. Moreover, the restriction of $\tau$ to $\mathcal{P}(\tau)$ is a $*$-automorphism. Furthermore, for a unital CP map $\tau$, they showed that $\mathcal{P}(\tau)=\mathcal{P}(\tau^n)$ for every $n\in\mathbb{N}$. We now discuss a few additional results that will be useful in characterizing eventually EB unital quantum channels.

 \begin{prop}[\cite{BD}]\label{EMD:thm1}
	For any unital quantum channel $\tau$ on $\mathbb{M}_d$, there exists $U\in \mathbb{U}(d)$ such that $\text{Ad}_U$ leaves $\mathcal{P}(\tau)^{\perp}$ invariant and $\tau(X)=UXU^*$ for all $X\in\mathcal{P}(\tau)$. Consequently, for each $n\in\mathbb{N}$, the channel $\text{Ad}_{{U^*}^n}\circ\tau^n$ fixes the peripheral space $\mathcal{P}(\tau)$ pointwise and $\lim\limits_{n\to\infty}\text{Ad}_{{U^*}^n}\circ\tau^n=\mathbb{E}_{\mathcal{P}(\tau)}$.
\end{prop}

\begin{prop}\label{EMD:lem9}
	For any unital quantum channel $\tau$ on $\mathbb{M}_d$, the commutant of the center of its peripheral space, $Z(\mathcal{P}(\tau))'$, is a reducing subspace for $\tau$.
\end{prop}

\begin{proof}
	Let $\{P_k\}_{k=1}^r$ denote the minimal central projections of $\mathcal{P}(\tau)$. We know that the restriction $\tau|_{\mathcal{P}(\tau)}$ acts as a $*$-automorphism. Therefore, there exists a permutation $\pi \in S_r$ such that $\tau(P_k) = P_{\pi(k)}$ for all $k$. Since each $P_k$ belongs to the multiplicative domain of $\tau$, $\tau(P_k X P_l) = P_{\pi(k)} \tau(X) P_{\pi(l)}$ for all $X \in \mathbb{M}_d$. The commutant naturally decomposes into the block-diagonal structure $Z(\mathcal{P}(\tau))' = \bigoplus_{k=1}^r P_k \mathbb{M}_d P_k$. The multiplicative relation ensures that $\tau$ maps each diagonal block $P_k \mathbb{M}_d P_k$ into $P_{\pi(k)} \mathbb{M}_d P_{\pi(k)}$, confirming that $Z(\mathcal{P}(\tau))'$ is invariant under $\tau$. 
	Further, for $k \neq l$, $\tau$ maps the off-diagonal blocks $P_k \mathbb{M}_d P_l$ into $P_{\pi(k)} \mathbb{M}_d P_{\pi(l)}$, yielding $(Z(\mathcal{P}(\tau))')^\perp = \bigoplus_{k \neq l} P_k \mathbb{M}_d P_l$ is preserved under $\tau$. Hence, $Z(\mathcal{P}(\tau))'$, is a reducing subspace for $\tau$.
\end{proof}

\begin{prop}\label{EMD:prop3}
	Let $\tau$ be a unital $2$-copositive channel on $\mathbb{M}_d$. Then, the peripheral space $\mathcal{P}(\tau)$ is a commutative $C^*$-algebra.
\end{prop}

\begin{proof}
	Since the peripheral space $\mathcal{P}(\tau)$ is a $C^*$-subalgebra of $\mathbb{M}_d, \;\mathcal{P}(\tau)$ is unitarily equivalent to $\bigoplus_{k=1}^r \mathbb{M}_{n_k} \otimes I_{m_k}$, where $\sum_{k=1}^r n_k m_k = d$. We must show that $n_k = 1$ for all $k$. If possible let, $n_k \geq 2$ for some $k$. Let $\{u_1,u_2\}$ be an orthonormal set for $\mathbb{C}^{n_k}$, and define the matrices $X_{ij} = |u_i\rangle\langle u_j| \otimes I_{m_k} \in \mathcal{P}(\tau)$ for $i,j\in\{1,2\}$. By Proposition \ref{EMD:thm1}, there exists a unitary $U \in \mathbb{M}_d$ such that $\tau(X_{ij}) = U X_{ij} U^*$ for all $i, j$.
	
	By hypothesis, $\tau$ is $2$-copositive, that is, the map $(\tau \circ T) \otimes \text{id}_{2}$ is positive. Applying this map to the positive semi-definite operator $\Sigma = \sum_{i,j=1}^2 X_{ij} \otimes |e_i\ra\la e_j|$ (where $\{e_1,e_2\}$ is a standard orthonormal basis of $\mathbb{C}^2$) yields
	\[
	((\tau \circ T) \otimes \text{id}_{2})(\Sigma) = \sum_{i,j} \tau(X_{ji}) \otimes |e_i\ra\la e_j| = (U \otimes I_d) \left( \sum_{i,j} X_{ji} \otimes |e_i\ra\la e_j| \right) (U \otimes I_d)^*.
	\]
	Therefore, the operator $\sum_{i,j=1}^2 X_{ji} \otimes |e_i\ra\la e_j|$ must be positive semi-definite. However, it essentially acts as the swap operator and admits $-1$ as an eigenvalue. Indeed, taking $v \in \mathbb{C}^{m_k}$ and distinct indices $p \neq q$, the vector $ (u_p \otimes v) \otimes e_q - (u_q \otimes v) \otimes e_p $ clearly serves as an eigenvector corresponding to $-1$. This contradicts the positivity of $\sum_{i,j=1}^2 X_{ji} \otimes |e_i\ra\la e_j|$. Thus, we must have $n_k = 1$ for all $k$. It follows that $\mathcal{P}(\tau)$ is unitarily equivalent to $\bigoplus_{k=1}^r \mathbb{C} I_{m_k},$ confirming it is commutative.
\end{proof}

\begin{eg}\label{commutative peripheral space but not EB}
	 Consider $\tau=\lambda\operatorname{id}_d+(1-\lambda)\Delta_d$ on $\mathbb{M}_d$ where $\lambda\in(0,1)$. The only possible peripheral eigenvalue of $\tau$ is $1$ and the peripheral space is actually the diagonal algebra, which is commutative. But $\tau$ is not $2$-copositive for any $\lambda$ as discussed in Example \ref{example in a nbd of expectation}. This shows that the converse of Proposition \ref{EMD:prop3} is not necessarily true for $d\ge 2$. In fact, $\tau$ is not eventually $2$-copositive also because $\tau^n=\lambda^n \operatorname{id}_d+(1-\lambda^n)\Delta_d$ is not $2$-copositive for any $n$. This demonstrates that the commutativity of the peripheral space alone is not sufficient to conclude eventually $2$-copositive (and hence eventually EB).
\end{eg}

\begin{rmk}
	Rahaman et al. \cite{RJP} established that the stabilized multiplicative domain of any unital PPT channel is commutative. Since, for unital quantum channels, the stabilized multiplicative domain coincides with the peripheral space (see \cite{BKT, MR}), their result equivalently shows that the peripheral space of any unital PPT channel is commutative. Proposition \ref{EMD:prop3} strengthens this observation: we show that mere \(2\)-copositivity of the map is sufficient to ensure this commutativity. 
	We discuss the following example to show that the set of PPT channels is properly contained in the set of $2$-copositive channels.
\end{rmk}

\begin{eg}\label{EMD:eg3}
	Let $d>2$ and consider the channel $\tau(X) = \lambda X^T + (1-\lambda)\delta_d(X)$. For any parameter $\lambda \in \big[ -\frac{1}{2d-1}, -\frac{1}{d^2-1} \big)$, this map is $2$-copositive but not PPT. For details, we refer the reader to Example \ref{example in appendix}.
\end{eg}

\begin{prop}\label{EMD:prop4}
	The range of every unital $2$-copositive channel is contained in the commutant of its peripheral space.
\end{prop}

\begin{proof}
	Let $\tau$ be a unital $2$-copositive channel on $\mathbb{M}_d$. By Proposition \ref{EMD:prop3}, $\mathcal{P}(\tau)$ is commutative, in particular,  $\mathcal{P}(\tau) = Z(\mathcal{P}(\tau))$. By Theorem \ref{EMD:thm1}, there exists $U \in \mathbb{U}(d)$ such that ${\mathcal{P}(\tau)}\subseteq\text{Fix}(\text{Ad}_U \circ \tau)$. Applying Theorem \ref{EMD:prop1} to the $2$-copositive channel $\text{Ad}_U \circ \tau$ yields $\text{Range}(\text{Ad}_U \circ \tau) \subseteq \mathcal{P}(\tau)'$. Furthermore, Lemmas \ref{EMD:lem5}, guarantees that $\mathcal{P}(\tau)'=Z(\mathcal{P}(\tau))'$ reduces $\text{Ad}_U \circ \tau$. Consequently, $\text{Ad}_U \circ \tau$ vanishes identically on the orthogonal complement $(\mathcal{P}(\tau)')^{\perp}$, forcing $\tau|_{(\mathcal{P}(\tau)')^{\perp}} \equiv 0$. Also, since $\mathcal{P}(\tau)'$ reduces $\tau$ (by Lemma \ref{EMD:lem9}), we conclude $\text{Range}(\tau) \subseteq \mathcal{P}(\tau)'$.
\end{proof}

Analogous to Definition~\ref{defn of eventually EB}, we formally introduce the following definition.

\begin{defn}
	A linear map $\tau$ on $\mathbb{M}_d$ is said to be \textit{eventually mixed twisted dephasing} if there exists $n_\circ\in\mathbb{N}$ such that $\tau^n$ is mixed twisted dephasing for all $n \ge n_\circ$. 
\end{defn}

\begin{rmk}
A unital quantum channel $\tau$ on $\mathbb{M}_d$ is eventually MTD if and only if there exists $n_0\in\mathbb{N}$ such that $\tau^{n_0}$ is MTD. Indeed, the necessity is immediate. Conversely, suppose that $\tau^{n_0}$ is MTD for some $n_0\in\mathbb{N}$. For $n\geq n_0$, we write $\tau^n=\tau^{n-n_0}\circ\tau^{n_0}$. By \cite[Theorem 3.9]{BD}, there exists $N\in\mathbb{N}$ such that $\tau^{n-n_0}$ is mixed unitary whenever $n\geq N+n_0$. Since the composition of a mixed unitary channel with an MTD is again MTD, it follows that $\tau^n$ is MTD for all $n\geq N+n_0$. Hence, $\tau$ is eventually MTD.
\end{rmk}

We are now ready to state the main theorem of this section. This characterizes eventually EB unital channels in terms of some commutation relations among their eigenvectors. Consequently, we establish that for a unital quantum channel, the notions of being eventually EB, eventually PPT, eventually $2$-copositive, and eventually MTD are equivalent. To establish this result, we make use of the interiority of the trace-preserving conditional expectation onto the peripheral space in the set of mixed twisted dephasing channels, as established in the previous section.

\begin{thm}\label{EMD:thm5}
	Let $\tau$ be a unital quantum channel on $\mathbb{M}_d$. Then, the following statements are equivalent:
	\begin{enumerate}
		\item $\tau$ is eventually MTD;
		\item $\tau$ is eventually EB;
		\item $\tau$ is eventually PPT;
		\item $\tau$ is eventually $2$-copositive;
		\item There exists an integer $n_0 \in \mathbb{N}$ such that $\text{Range}(\tau^{n_0}) \subseteq \mathcal{P}(\tau)'$;
        \item The intersection $\bigcap_{n=1}^\infty \text{Range}(\tau^n)$ is contained in $\mathcal{P}(\tau)'$;
        \item All eigenvectors of $\tau$ corresponding to  non-zero eigenvalues commute with all peripheral eigenvectors, that is, $\ker(\tau - \lambda\operatorname{id}_d)\subseteq \mathcal{P}(\tau)'$ for all $\lambda \neq 0$. 
	\end{enumerate}
\end{thm}

\begin{proof}
	  The implications (1) $\Rightarrow$ (2) $\Rightarrow$ (3) $\Rightarrow$ (4) and (5) $\Rightarrow$ (6) are immediate. We proceed to prove the remaining implications.

	   \noindent \textbf{(4) $\Rightarrow$ (5):} Assume the statement (4) holds. Then, there is $n_0\in\mathbb{N}$ such that $\tau^{n_0}$ is $2$-copositive. Proposition \ref{EMD:prop4} implies that $\text{Range}(\tau^{n_0}) \subseteq \mathcal{P}(\tau^{n_0})'$. By \cite[Theorem 3.1]{BKT}, the peripheral space is invariant under taking powers of $\tau$, that is, $\mathcal{P}(\tau^{n_0}) = \mathcal{P}(\tau)$, which yields the desired inclusion.

	     \noindent \textbf{(6) $\Rightarrow$ (7):} Let $X$ be an eigenvector satisfying $\tau(X) = \lambda X$ for some $\lambda \neq 0$. For any $n \in \mathbb{N}$, we can write $X = \lambda^{-n}\tau^n(X)$, yielding $X \in \text{Range}(\tau^n)$. Thus, $X \in \bigcap_{n=1}^\infty \text{Range}(\tau^n)$. Hypothesis (6) forces $X \in \mathcal{P}(\tau)'$, that is, $X$ commutes with all peripheral eigenvectors. 
	     
	     \noindent \textbf{(7) $\Rightarrow$ (1):} Assume the statement (7) holds. Firstly, we note that because peripheral eigenvectors themselves correspond to non-zero eigenvalues, they automatically commute with each other. This ensures the commutativity of $\mathcal{P}(\tau)$.  By Proposition 6.2 \ref{EMD:thm1}, there exists $U\in\mathbb{U}(d)$ such that $\sigma_n = \text{Ad}_{U^n} \circ \tau^n$ fixes the commutative $C^*$-subalgebra $\mathcal{P}(\tau)$ pointwise for every $n$ and converges to $\mathbb{E}_{\mathcal{P}(\tau)}$. Note, by Lemma \ref{EMD:lem5} and \ref{EMD:lem9}, that $\mathcal{P}(\tau)'=Z(\mathcal{P}(\tau))'$ is a reducing subspace for both $\tau$ and $\sigma_n$. The assumption that $\ker(\tau - \lambda\operatorname{id}_d) \subseteq \mathcal{P}(\tau)'$ for all $\lambda \neq 0$ implies that the restriction of $\tau$ to the orthogonal complement $(\mathcal{P}(\tau)')^{\perp}$ has no non-zero eigenvalues. Thus, there exists an integer $n_0$ such that $\tau^n|_{(\mathcal{P}(\tau)')^{\perp}} \equiv 0$, and hence $\sigma_n|_{(\mathcal{P}(\tau)')^{\perp}} \equiv 0$, for all $n \ge n_0$. This nilpotency forces $\text{Range}(\sigma_n) \subseteq \mathcal{P}(\tau)'$ for all $n \ge n_0$. Therefore, $\sigma_n$ is a sequence of unital quantum channels satisfying $\mathcal{P}(\tau) \subseteq \text{Fix}(\sigma_n) \cap \text{Range}(\sigma_n)'$ for all $n\geq n_\circ$ and converges to $\mathbb{E}_{\mathcal{P}(\tau)}$. Hence, by Theorem \ref{EMD:cor1}, $\sigma_n$ is MTD for all sufficiently large $n$. Since $\mathrm{MTD}(d)$ is closed under composition with mixed unitaries from either side, $\tau^n = \text{Ad}_{(U^n)^*} \circ \sigma_n$ is MTD for large $n$. Therefore, $\tau$ is eventually MTD.
\end{proof}
\begin{rmk}
    The nested sequence of subspaces $\text{Range}(\tau) \supseteq \text{Range}(\tau^2) \supseteq \cdots$ stabilizes as the underlying space is finite-dimensional. The resulting intersection $\bigcap_{n=1}^\infty \text{Range}(\tau^n)= \bigcap _{n=1}^{d^2}\text{Range}(\tau ^n)$ is the subspace spanned by the generalized eigenvectors corresponding to non-zero eigenvalues of $\tau$. Thus, condition (6) in Theorem \ref{EMD:thm5} is equivalent to requiring that every such generalized eigenvector belongs to the commutant $\mathcal{P}(\tau)'$.

 The   condition (7) in  Theorem \ref{EMD:thm5} restricts the commutativity requirement to eigenvectors associated with non-zero eigenvalues. One may ask whether this condition is automatic  for eigenvectors corresponding to zero eigenvalue, that is for matrices in the kernel of $\tau $. But this is not the case as Example \ref{EMD:eg6} shows. The reason for this is that the kernel is not represented in $\bigcap_{n=1}^\infty \text{Range}(\tau^n).$
\end{rmk}

\begin{rmk}
Let $\mathrm{EEB}(d)$ denote the set of all eventually EB unital quantum channels on $\mathbb{M}_d$. A set $\mathcal{S}$ is said to be \textit{absorbing } in $\mathrm{EEB}(d)$ if, for every $\tau\in\mathrm{EEB}(d)$, there exists $n_0\in\mathbb{N}$ such that $\tau^n\in\mathcal{S}$ for all $n\geq n_0$. By Theorem \ref{EMD:thm5}, $\mathrm{MTD}(d)$ is absorbing in $\mathrm{EEB}(d)$. Recall that $\mathrm{MTD}(d)$ is convex and closed under composition with unitary channels. In fact, it is the minimal absorbing set in $\mathrm{EEB}(d)$ exhibiting these properties. To see this, observe that any absorbing set in $\mathrm{EEB}(d)$ must contain the idempotent channels $\beta_{U,U}$ for all $U \in \mathbb{U}(d)$. If this set is additionally closed under composition with unitary channels, it captures every twisted dephasing channel via the relation $\operatorname{Ad}_W \circ \beta_{U,U} = \beta_{WU,U}$. Therefore, convexity forces the inclusion of their convex combinations, which is $\mathrm{MTD}(d)$.
\end{rmk}

\begin{rmk}\label{universal bound}
      For $d \ge 2$, there is no integer $N\in\mathbb{N}$ such that $\tau^N$ is $2$-copositive for all $\tau\in\mathrm{EEB}(d)$.
      \begin{proof}
          Let $N\in\mathbb{N}$ and $\lambda\in((d+1)^{-\frac{1}{N}},1)$. Consider the depolarizing channel $\tau=\lambda\text{id}_d+(1-\lambda)\delta_d$. Check that $\tau^n=\lambda^n\text{id}_d+(1-\lambda^n)\delta_d$ for $n\geq 1$. By Example \ref{example in appendix}, $\tau\in\mathrm{EEB}(d)$, but its $N$-th power, $\tau^N$ is not $2$-copositive as $\lambda^N>\frac{1}{d+1}$.
      \end{proof}
  \end{rmk}

We will now discuss some consequences. Recall that a unital quantum channel $\tau$ on $\mathbb{M}_d$ is primitive (that is, irreducible with $\mathrm{spec}(\tau) \cap \mathbb{T} = \{1\}$) if and only if $\mathcal{P}(\tau) = \mathbb{C}I_d$ (see \cite{BKT, MR}).

\begin{cor}\label{primitive}
	Any unital primitive channel on $\mathbb{M}_d$ is eventually MTD.
\end{cor}

\begin{proof}
	Since the peripheral space of a unital primitive channel is trivial, the conclusion follows from Theorem \ref{EMD:thm5} (part (5) $\Rightarrow$ (1)).
\end{proof}

The following example shows that an eventually MTD channel need not be primitive.

\begin{eg}\label{EMD:eg2}
	Let $m_1, m_2$ be two natural numbers and $\tau,\sigma$ be two unital primitive channels on $\mathbb{M}_{m_1}$ and $\mathbb{M}_{m_2}$ respectively. Then, by Corollary \ref{primitive} and Theorem \ref{EMD:thm5} (part $(1) \Rightarrow (4)$), there exists $n\in\mathbb{N}$ such that both $\tau^n$ and $\sigma^n$ are $2$-copositive. Take $\varphi=\tau\oplus \sigma$. Then, $\varphi^{n}$ is also $2$-copositive. Therefore, again by Theorem \ref{EMD:thm5} (part $(4) \Rightarrow (1)$), $\varphi$ is eventually MTD. But $\varphi$ is not primitive because direct sum of primitive channels is not primitive.
\end{eg}

The following corollary strengthens Theorem~\ref{Rahaman's result}.

\begin{cor}
	Let $\tau$ be a unital $2$-copositive channel on $\mathbb{M}_d$. Then, $\tau$ is eventually MTD.
\end{cor}

\begin{proof}
	The proof is immediate by Theorem \ref{EMD:thm5} (part $(4)\Rightarrow (1)$).
\end{proof}

The following example shows that an eventually MTD channel need not be $2$-copositive.

\begin{eg}\label{EMD:eg6}
	Let $d\geq 2$ be a natural number. Consider the unital quantum channel $\tau=\lambda\text{id}_d+(1-\lambda)\delta_d$ for $\lambda\in \left(\frac{1}{d+1},1\right)$. Since $\tau$ is primitive, by Corollary \ref{primitive}, $\tau$ is eventually MTD (see Example \ref{example in appendix} for an alternative argument). But $\tau$ is not $2$-copositive.  For details, we refer the reader to Example \ref{example in appendix}.
\end{eg}

  \section{Eventually EB  unital channels in continuous time}\label{continuous time}
   
  In this section, we characterize continuous semigroups of unital quantum channels that become EB after some finite time and establish the equivalence of the properties of being eventually EB, eventually PPT, eventually $2$-copositive, primitive, and eventually MTD.
  
  \begin{defn}
  	A family of linear maps $\{\tau_t\}_{t\geq 0}$ on $\mathbb{M}_d$ is said to be a \textit{one-parameter semigroup} if $\tau_0=\operatorname{id}_d$, $\tau_t\circ\tau_s=\tau_{s+t}$ for all $t,s\geq 0$, and the map $t\mapsto\tau_t$ is continuous. A one-parameter semigroup $\{\tau_t\}_{t\geq 0}$ is said to be a \textit{quantum dynamical semigroup} if $\tau_t$ is CP for every $t$. Moreover, it is said to be a \textit{quantum Markov semigroup} if $\tau_t$ is UCP for every $t$. 
  \end{defn}

  Let $\{\tau_t\}_{t\geq 0}$ be a one-parameter semigroup of linear maps on $\mathbb{M}_d$. Then, the map $\mathcal{L}$ on $\mathbb{M}_d$ defined by
  \begin{align*}
  	\mathcal{L}(X)=\lim_{t\rightarrow 0+}\frac{\tau_t(X)-X}{t}\quad\text{for all } X\in\mathbb{M}_d,
  \end{align*}
  is well-defined, and generates the semigroup via
  \begin{align*}
  	\tau_t(X)=e^{t\mathcal{L}}(X)=\sum_{n=0}^\infty\frac{t^n\mathcal{L}^n(X)}{n!}\quad\text{for all }X\in\mathbb{M}_d\text{ and }t\ge 0.
  \end{align*}
  This map $\mathcal{L}$ is called the \textit{infinitesimal generator} of the semigroup $\{\tau_t\}_{t\geq 0}$. The generator of a quantum dynamical semigroup (or, quantum Markov semigroup) is characterized by the famous GKLS theorem. The reader may refer to \cite{GKS,L} for details.

   \begin{defn}[Peripheral Poisson boundary]
  	For a quantum Markov semigroup $\{\tau_t = e^{t\mathcal{L}}\}_{t \ge 0}$ on $\mathbb{M}_d$, the \textit{peripheral Poisson boundary} of the generator $\mathcal{L}$ (or, of the semigroup), denoted by $\mathcal{P}(\mathcal{L})$, is defined as
  	\[
  	\mathcal{P}(\mathcal{L}) = \text{span}\{X \in \mathbb{M}_d : \mathcal{L}(X) = iaX \text{ for some } a \in \mathbb{R}\}.
  	\]
  \end{defn}
  
  Note that this subspace is unital, self-adjoint preserving, and contained in the peripheral space of $\tau_t$, that is, $\mathcal{P}(\mathcal{L})\subseteq\mathcal{P}(\tau_t)$ for all $t\geq 0$. In fact, it is known that $\mathcal{P}(\mathcal{L})=\mathcal{P}(\tau _t)$ for all $t>0$, but we will not need this fact. Moreover, if $\{\tau_t = e^{t\mathcal{L}}\}_{t \ge 0}$ is a quantum Markov semigroup of quantum channels on $\mathbb{M}_d$,
  
  then the peripheral Poisson boundary $\mathcal{P}(\mathcal{L})$ is a $C^*$-subalgebra of $\mathbb{M}_d$ (see  \cite[Lemma 4.5]{BD}). For further details on the peripheral Poisson boundary, we refer the reader to \cite{BA}.

  \begin{defn}[Primitive]
  	A quantum Markov semigroup $\{\tau_t\}_{t \ge 0}$ on $\mathbb{M}_d$ is said to be \textit{primitive} if there exists a unique, strictly positive density matrix $\rho \in \mathbb{M}_d$ such that $\tau_t^*(\rho) = \rho$ for all $t \ge 0$.
  \end{defn}

   \begin{rmk}
  	The study of irreducibility of positive maps was pioneered by Evans and H{\o}egh-Krohn \cite{EvaDE:HoeR78} and they established the non-commutative Perron-Frobenius theory for irreducible positive maps. Consequently, the irreducibility and ergodic properties of quantum dynamical semigroups were studied by Evans \cite{EvaDE77} and Frigerio \cite{FriA77}. Later, Sanz et al. \cite{SPWC} formally introduced the notion of primitivity for unital quantum channels in terms of irreducibility along with a trivial peripheral spectrum. This framework was subsequently extended to continuous-time quantum Markov semigroups \cite{BCGPY,RD}. For characterizations and sufficient operational conditions for primitivity, we refer the reader to \cite{BCGPY,Wo}. 
  \end{rmk}

   \begin{rmk}\label{EMD:rem1}
  	Let $\{\tau_t = e^{t\mathcal{L}}\}_{t \ge 0}$ be a quantum Markov semigroup of quantum channels on $\mathbb{M}_d$. Then, $\{\tau_t\}_{t\geq 0}$ is primitive if and only if $\text{ker}(\mathcal{L})=\mathbb{C}I_d$.
  \end{rmk}

  \begin{proof}
  	$``\Rightarrow"$. Let $\{\tau_t\}_{t\geq 0}$ be primitive. Then, $\text{ker}(\mathcal{L}^*)=\cap_{t\geq 0}\text{Fix}(\tau_t^*)=\mathbb{C}I_d$. Since $\text{dim}(\text{ker}(\mathcal{L}))=\text{dim}(\text{ker}(\mathcal{L}^*))=1$, we have $\text{ker}(\mathcal{L})=\mathbb{C}I_d$.
  	
  	$``\Leftarrow"$. Let $\text{ker}(\mathcal{L})=\mathbb{C}I_d$. Since $\text{dim}(\text{ker}(\mathcal{L^*}))=\text{dim}(\text{ker}(\mathcal{L}))=1$, we have $\text{ker}(\mathcal{L}^*)=\mathbb{C}I_d$. Therefore, $\cap_{t\geq 0}\text{Fix}(\tau_t^*)=\text{ker}(\mathcal{L}^*)=\mathbb{C}I_d$. Hence, $\{\tau_t\}_{t\geq 0}$ is primitive
  \end{proof}

  	Formally, a one-parameter semigroup $\{\tau_t\}_{t\geq 0}$ is said to be \textit{eventually EB} (respectively, \textit{eventually MTD}, \textit{eventually PPT}, or \textit{eventually $2$-copositive}) if there exists $t_0>0$ such that $\tau_t$ is EB (respectively, MTD, PPT, or $2$-copositive) for every $t\geq t_0$. As in the discrete case, a quantum Markov semigroup $\{\tau_t\}_{t\geq 0}$ of quantum channels is eventually EB (respectively, eventaully MTD, eventually PPT, or eventually $2$-copositive) if and only if $\tau_{t_0}$ is EB (respectively, MTD, PPT, or $2$-copositive) for some $t_0>0$.

  We are now ready to state the main theorem of this section, which characterizes eventual EB property for quantum Markov semigroups of quantum channels and establishes the equivalence between being primitive, eventually EB, eventually PPT, eventually $2$-copositive, and eventually MTD. 
 The equivalence between primitivity and the eventual EB property for quantum Markov semigroups was first established by  Hanson, Rouz\'e, and Fran\c ca \cite[Theorem 2.1]{HRS}.  By framing primitivity through the trivial peripheral Poisson boundary, our framework not only recovers this result for semigroup of unital channels but expands it to include the eventual PPT, eventual $2$-copositive, and eventual MTD properties. Thus, rather than merely breaking entanglement after some time, such semigroups ultimately enter into the convex hull of twisted dephasing channels. Furthermore, our approach involving this boundary offers a different perspective and emphasizes that much like the discrete
setting: the interiority of $\delta_d$ within $\mathrm{MTD}(d)$ (Theorem \ref{EMD:cor2}) is sufficient to guarantee this eventual mixed twisted  behaviour of primitive semigroups in continuous time. 
  
   \begin{thm}\label{EMD:thm4}
  	Let $\{\tau_t = e^{t\mathcal{L}}\}_{t \ge 0}$ be a quantum Markov semigroup of quantum channels on $\mathbb{M}_d$. Then, the following statements are equivalent:
  	\begin{enumerate}
  		\item $\{\tau_t\}$ is eventually MTD;
  		\item $\{\tau_t\}$ is eventually EB;
  		\item $\{\tau_t\}$ is eventually PPT;
  		\item $\{\tau_t\}$ is eventually $2$-copositive;
  		\item The peripheral Poisson boundary is trivial, that is, $\mathcal{P}(\mathcal{L})=\mathbb{C}I_d$;
  		\item $\{\tau_t\}$ is primitive;
  		\item $\lim\limits_{t\to\infty}\tau_t=\delta_d$.
  	\end{enumerate}
  \end{thm}
  
  \begin{proof}
  	The assertions $(1)\Rightarrow(2)\Rightarrow(3)\Rightarrow(4)$ are immediate. We prove the remaining implications below.
  	
  	  \noindent\textbf{(4) $\Rightarrow$ (5):} Suppose there exists a time $t_0 > 0$ such that $\tau_{t_0}$ is $2$-copositive. By Proposition \ref{EMD:prop4}, the range of $\tau_{t_0}$ must be contained within the commutant of its peripheral space, that is, $\text{Range}(\tau_{t_0}) \subseteq \mathcal{P}(\tau_{t_0})'$. Because $\tau_{t_0} = e^{t_0\mathcal{L}}$ is invertible linear map, its range is the whole of $\mathbb{M}_d$. This forces $\mathcal{P}(\tau_{t_0})' = \mathbb{M}_d$, which implies $\mathcal{P}(\tau_{t_0}) = \mathbb{C}I_d$. Since the peripheral Poisson boundary of the generator satisfies $\mathcal{P}(\mathcal{L}) \subseteq \mathcal{P}(\tau_{t_0})$, we obtain $\mathcal{P}(\mathcal{L}) = \mathbb{C}I_d$.

  	  \noindent \textbf{(5) $\Rightarrow$ (6):} Assume $\mathcal{P}(\mathcal{L}) = \mathbb{C}I_d$. Since $\tau_t$'s are unital, $\mathcal{L}(I_d)=0$. Therefore, $\mathbb{C}I_d\subseteq\ker(\mathcal{L})\subseteq \mathcal{P}(\mathcal{L})=\mathbb{C}I_d$. This forces $\ker(\mathcal{L})=\mathbb{C}I_d$. Hence, by Remark \ref{EMD:rem1}, $\{\tau_t\}_{t\geq 0}$ is primitive.

  	  \noindent \textbf{(6) $\Rightarrow$ (7):} Let $\{\tau_t\}_{t\ge0}$ be primitive. We first establish that $\mathcal{L}$ admits no non-zero purely imaginary eigenvalues. Suppose $\mathcal{L}(X) = iaX$ for some $a \in \mathbb{R}$ and non-zero $X \in \mathbb{M}_d$. Then, $\tau_t(X)=e^{t\mathcal{L}}(X)=e^{iat}X$, yielding $X\in\mathcal{P}(\tau_t)$ for all $t\ge 0$. Since $\tau_t$ restricts to its peripheral space $\mathcal{P}(\tau_t)$ is a $*$-automorphism, we have $\tau_t(X^*X)=\tau_t(X)^*\tau_t(X)=X^*X$, implying $X^*X \in\bigcap_{t\ge0}\text{Fix}(\tau_t)= \ker(\mathcal{L})$, which equals $\mathbb{C}I_d$ by Remark \ref{EMD:rem1}. Therefore, by the polar decomposition, $X$ is a non-zero scalar multiple of a unitary, so it possesses at least one non-zero eigenvalue. Further, since unital $*$-automorphism preserves spectrum, we have $\mathrm{spec}(X) = \mathrm{spec}(e^{iat}X) = e^{iat}\mathrm{spec}(X)$ for all $t \ge 0$. Because $X$ has a non-zero eigenvalue and its spectrum is finite, we must have $e^{iat} =1$ for all $t\ge0$, forcing $a = 0$. Hence, $\mathcal{L}$ has no non-zero purely imaginary eigenvalues.

  	   Next, we show that $\lim_{t \to \infty} \tau_t$ exists. We decompose the space $\mathbb{M}_d$ into the generalized eigenspaces of $\mathcal{L}$, namely $\mathbb{M}_d = \bigoplus_{k=1}^m \ker(\mathcal{L}-\lambda_k\text{id}_d)^{n_k}$. For any $X \in \ker(\mathcal{L}-\lambda_k\text{id}_d)^{n_k}$, the norm of semigroup acting on $X$ expands as
  	  \[
  	  \|\tau_t(X)\| = \left\|e^{t\lambda_k\text{id}_d+t(\mathcal{L}-\lambda_k\text{id}_d)}(X)\right\| = e^{t\Re(\lambda_k)}\left\| \sum_{j=0}^{n_k-1} \frac{t^j}{j!} (\mathcal{L}-\lambda_k\text{id}_d)^j(X)\right\|.
  	  \]
  	  Because each $\tau_t$ is unital, it is contractive, meaning $\|\tau_t(X)\| \le \|X\|$ for all $t \ge 0$. This uniform boundedness refutes the case $\Re(\lambda_k) > 0$, which would induce unbounded exponential growth. If $\Re(\lambda_k) < 0$, the expression clearly vanishes as $t \to \infty$. If $\Re(\lambda_k) = 0$, we must have $n_k = 1$; otherwise, the terms $t^j$ (for $j \ge 1$) would force $\|\tau_t(X)\| \to \infty$. Thus, in this case $n_k=1$, yielding $\lambda_k$ to be purely imaginary eigenvalue of $\mathcal{L}$. By our earlier deduction, $\lambda_k$ must be $0$ and the corresponding eigenvector $X = cI_d$ for some $c \in \mathbb{C}$ (by Remark  \ref{EMD:rem1}). This confirms that $\lim_{t\to\infty}\tau_t(X)$ converges for all $X$.

  	    Let $E=\lim_{t \to \infty} \tau_t$. Finally, we show that $E= \delta_d$. For any $X=E(Y) \in \text{Range}(E)$ and $s\ge0$, $\tau_s(X) = \lim_{t \to \infty} \tau_{s+t}(Y) = E(Y) = X$. Hence, $\mathcal{L}(X) = 0$, implying $\text{Range}(E) \subseteq \ker(\mathcal{L})$. Since $\{\tau_t\}_{t\ge0}$ is primitive, by Remark \ref{EMD:rem1}, $\ker(\mathcal{L})= \mathbb{C}I_d$, yielding $\text{Range}(E)=\mathbb{C}I_d$. Therefore, there is a linear functional $\phi:\mathbb{M}_d\to\mathbb{C}$ such that $E(X)=\phi(X)I_d$ for all $X\in \mathbb{M}_d$. Since $E$ is trace-preserving, applying trace on the both sides we obtain $\phi(X)=\frac{\text{Tr}(X)}{d}$ for all $X\in \mathbb{M}_d$, concluding $E=\delta_d$.

  	    \noindent \textbf{(7) $\Rightarrow$ (1):} Assume $\lim_{t \to \infty} \tau_t = \delta_d$. By Theorem \ref{EMD:cor2}, there exits $t_0 > 0$ such that $\tau_{t}$ is MTD for every $t\geq t_0$. Therefore, $\{\tau_t\}_{t\geq 0}$ is eventually MTD.
  \end{proof}

   \section{Examples}\label{appendix}
   This section presents an instructive class of examples.
   
  \begin{eg}\label{example in appendix}
  	Let $d\geq 2$ be a natural number, and for $\lambda\in\mathbb{R}$, define $\tau_\lambda=\lambda\,\text{id}_d+(1-\lambda)\delta_d$. Then, the following statements hold:
  	\begin{enumerate}
  		\item $\tau_\lambda$ is $2$-positive if and only if $\lambda\in\left[-\frac{1}{2d-1},1\right]$;
  		\item $\tau_\lambda$ is $2$-copositive if and only if $\lambda\in\left[-\frac{1}{d-1},\frac{1}{d+1}\right]$;
  		\item $\tau_\lambda$ is CP if and only if $\lambda\in \left[-\frac{1}{d^2-1},1\right]$;
  		\item $\tau_\lambda$ is coCP if and only if $\lambda\in\left[-\frac{1}{d-1},\frac{1}{d+1}\right]$;
  		\item $\tau_\lambda$ (hence $\tau_\lambda\circ T$) is EB if and only if $\lambda\in\left[-\frac{1}{d^2-1},\frac{1}{d+1}\right]$;
  		\item $\tau_\lambda$ (hence $\tau_\lambda\circ T$) is MTD if and only if $\lambda\in\left[-\frac{1}{d^2-1},\frac{1}{d+1}\right]$;
  	\end{enumerate}
  \end{eg}
  
  \begin{proof}
  	We will prove statements (1), (2) and (6). Statements (3) and (4) follow directly from Theorem \ref{CJ isomorphism}. For a proof of statement (5), the reader may refer to \cite{HH}.

   \noindent\textbf{Proof of (1):} Let $\xi \in \mathbb{C}^d \otimes \mathbb{C}^2$. By the Schmidt decomposition, we write $\xi = a(x \otimes u) + b(y \otimes v)$ where $\{x, y\} \subseteq\mathbb{C}^d$ and $\{u, v\} \subseteq\mathbb{C}^2$ are orthonormal sets and $a, b \ge 0$. Evaluating $X_{\lambda,\xi} = (\tau_\lambda\otimes\text{id}_{2})(|\xi\ra\la\xi|)$ yields
  	\[
  	X_{\lambda,\xi} = \lambda|\xi\ra\la\xi| + \frac{1-\lambda}{d}(I_d\otimes \rho),
  	\]
  	where $\rho = a^2|u\ra\la u| + b^2|v\ra\la v|\in \mathbb{M}_2$. Define the subspace $W = \text{span}\{x \otimes u, y \otimes v\}$. Since $\text{Range}(|\xi\ra\la\xi|) \subseteq W$ and $W$ is invariant under $I_d\otimes \rho$, we have that $X_{\lambda,\xi} \ge 0$ if and only if its restrictions to both $W$ and $W^\perp$ are positive. On $W^\perp$, $X_{\lambda,\xi}$ acts as $\frac{1-\lambda}{d}(\rho \otimes I_d)$, which is positive if and only if $\lambda \le 1$. On $W$, the matrix representation of $X_{\lambda,\xi}$ with respect to the basis $\{x \otimes u, y \otimes v\}$ is
  	\[
  	X_{\lambda,\xi}|_W = \begin{bmatrix}
  		a^2 c_\lambda & a b \lambda\\
  		a b \lambda & b^2 c_\lambda
  	\end{bmatrix}, \quad \text{where} \quad c_\lambda= \frac{\lambda(d-1)+1}{d}.
  	\]
  	This block is positive semi-definite if and only if $c_\lambda \ge 0$ and its determinant is non-negative, yielding $a^2b^2(c_{\lambda}^2 - \lambda^2) \ge 0$. For strictly positive $a, b$, this implies that $(c_\lambda+\lambda)(c_\lambda-\lambda)\ge0\Leftrightarrow ((2d-1)\lambda+1)(1-\lambda)\ge0$. This enforces $-\frac{1}{2d-1}\le \lambda\le 1$. Therefore, $X_{\lambda,\xi} \ge 0$ for any $\xi$ if and only if $\lambda \in \left[-\frac{1}{2d-1},1\right]$.

  	\noindent\textbf{Proof of (2):} Let $\xi \in \mathbb{C}^d \otimes \mathbb{C}^2$. By the Schmidt decomposition, we write $\xi = a(x \otimes u) + b(y \otimes v)$ where $\{x, y\} \subseteq\mathbb{C}^d$ and $\{u, v\} \subseteq \mathbb{C}^2$ are orthonormal sets and $a, b \ge 0$. Evaluating $Y_{\lambda,\xi} = ((\tau_\lambda\circ T)\otimes \text{id}_{2})(|\xi\ra\la\xi|)$ yields
  	\[
 	Y_{\lambda,\xi} = \lambda(T\otimes \text{id}_{2})(|\xi\ra\la\xi|) + \frac{1-\lambda}{d}(I_d\otimes \rho),
  	\]
  	where $\rho = a^2|u\ra\la u| + b^2|v\ra\la v|\in \mathbb{M}_2$. Expanding $|\xi\ra\la\xi|$ and evaluating the trace on the first tensor component, we observe 
  	\[(T\otimes \text{id}_{2})(|\xi\ra\la\xi|) = a^2 |\overline{x}\ra\la \overline{x}|\otimes|u\ra\la u|+ b^2 |\overline{y}\ra\la \overline{y}| \otimes |v\ra\la v|+ ab |\overline{y}\ra\la \overline{x}|\otimes |u\ra\la v| + ab |\overline{x}\ra\la \overline{y}|\otimes |v\ra\la u|.\]Define the subspaces $W = \text{span}\{ \overline{x}\otimes u, \, \overline{y}\otimes v\}\text{ and }W' = \text{span}\{\overline{y}\otimes u, \, \overline{x}\otimes v\}$. Because $\text{Range}((T\otimes \text{id}_{2})(|\xi\ra\la\xi|)) \subseteq W\oplus W'$ and $W$, $W'$ are invariant under $I_d \otimes \rho$, we have $X_{\lambda,\xi} \ge 0$ if and only if its restriction on all three subspaces $W$, $W'$ and $(W\oplus W')^\perp$ are positive. On $(W\oplus W')^\perp$, $Y_{\lambda,\xi}$ acts as $\frac{1-\lambda}{d}(I_d \otimes \rho)$, which is positive if and only if $\lambda \le 1$. 
  	
  	On $W$, the matrix representation of $Y_{\lambda,\xi}$ with respect to the basis $\{ \overline{x}\otimes u, \, \overline{y}\otimes v\}$ is
  	\[
  	Y_{t,\xi}|_W = \begin{bmatrix}
  		a^2 c_\lambda & 0 \\
  		0 & b^2 c_\lambda
  	\end{bmatrix}, \quad \text{where} \quad c_\lambda= \frac{\lambda(d-1)+1}{d}.
  	\]
  	This block is positive semi-definite if and only if $c_\lambda\ge 0$, yielding $ \lambda \ge -\frac{1}{d-1}$. 
  	
  	Further, on $W'$, the matrix representation of $Y_{\lambda,\xi}$ with respect to the basis $\{\overline{y}\otimes u, \, \overline{x}\otimes v\}$ is
  	\[Y_{\lambda,\xi}\vert{}_{W'} = \begin{bmatrix} \frac{a^2 (1-\lambda)}{d} & ab\lambda\\ ab\lambda & \frac{b^2(1-\lambda)}{d} \end{bmatrix}.\]
  	This block is positive semi-definite if and only if its diagonal entries are non-negative (which confirms $\lambda  \le 1$), and its determinant is non-negative, yielding $a^2b^2((\frac{1-\lambda}{d})^2 - \lambda^2) \ge 0$. For strictly positive $a, b$, this implies that $((1-\lambda)^2-(d\lambda)^2)\ge0\Leftrightarrow\left(1 + (d-1)\lambda\right)\left(1 - (d+1)\lambda\right) \ge 0$. From our analysis on $W$, we already established that $1 + (d-1)\lambda \ge 0$. Therefore, to have non-negativity, the second factor must also be non-negative. This forces $\lambda \le \frac{1}{d+1}$. Therefore, $Y_{\lambda,\xi} \ge 0$ for any $\xi$ if and only if $\lambda \in \left[-\frac{1}{d-1}, \frac{1}{d+1}\right]$.

  	\noindent\textbf{Proof of (6):} By Lemma \ref{EMD:lem2}, for any $U\in\mathbb{U}(d)$, we have the formula
  	\begin{equation}\label{EMD:eq5}
  		\int_{\mathbb{U}(d)} \beta_{V, VU} \, \mathrm{d}V = \mu(U)\,\text{id}_d + \big(1 - \mu(U)\big)\delta_d,
  	\end{equation}
  	where the weight is $\mu(U) =\frac{1}{d^2-1}(\text{Tr}(\Delta_d\circ\text{Ad}_{U^*})-1)= \frac{1}{d^2-1} \big( \sum_{k=1}^d |\la e_k, U e_k \ra|^2 - 1 \big)$. Because $0 \le |\la e_k, U e_k \ra|^2 \le 1$, the continuous map $\mu: \mathbb{U}(d) \to \mathbb{R}$ necessarily has its range contained in $\left[-\frac{1}{d^2-1}, \frac{1}{d+1}\right]$. We claim that $\text{Range}(\mu)=\left[-\frac{1}{d^2-1}, \frac{1}{d+1}\right]$. Evaluating $\mu$ at $U = I_d$ achieves the maximum $\frac{1}{d+1}$, while choosing $U\in \mathbb{M}_d$ to be the cyclic permutation matrix (which possesses a zero diagonal) achieves the minimum $-\frac{1}{d^2-1}$. Therefore, by the path-connectedness of $\mathbb{U}(d)$, we have $\text{Range}(\mu)=\left[-\frac{1}{d^2-1}, \frac{1}{d+1}\right]$, yielding $\tau_\lambda$ is MTD for every $\lambda\in\left[-\frac{1}{d^2-1}, \frac{1}{d+1}\right]$, by Equation \ref{EMD:eq5}.

  	Conversely, let $\tau_\lambda$ be mixed twisted dephasing. Then, $\tau_\lambda$ admits a decomposition $\tau_\lambda= \sum_{k=1}^n p_k \beta_{U_k, V_k}$ where $\{U_k,V_k:1\leq k\leq n\}\subseteq\mathbb{U}(d)$, and $p_k\geq 0$ with $\sum_{k=1}^n p_k=1$. We apply the twirling operation $\mathscr{T}(\varphi) = \int_{\mathbb{U}(d)} \text{Ad}_W \circ \varphi \circ \text{Ad}_{W^*} \, \mathrm{d}W$. Utilizing Lemma \ref{EMD:lem2} again, its action on the individual $\beta_{U,V}$ is
  	\[
  	\mathscr{T}(\beta_{U, V}) = \int_{\mathbb{U}(d)} \beta_{W, W U^* V} \, \mathrm{d}W = \mu(U^* V)\,\text{id}_d + \big(1 - \mu(U^* V)\big)\delta_d.
  	\]
  	Since $\tau_\lambda$ is unitarily covariant by its definition, $\mathscr{T}(\tau_\lambda) = \tau_\lambda$. Applying $\mathscr{T}$ to the assumed convex decomposition of $\tau_\lambda$ forces $\lambda= \sum_{k=1}^n p_k \mu(U_k^* V_k)$. Because every coefficient $\mu(U_k^* V_k)$ is bounded within $\left[-\frac{1}{d^2-1}, \frac{1}{d+1}\right]$, their convex combination $\lambda$ must also reside inside this interval.
  \end{proof}

  \section{Acknowledgments}
  Bhat gratefully acknowledges funding from ANRF (India) through J C Bose Fellowship No. JBR/2021/000024. The research of the second-named author is supported by the NBHM Postdoctoral Fellowship, Department of Atomic Energy, Government of India (File No: 0204/9(34)/2026/R\&D-II/4857). The authors used AI solely to refine the English and create Figures \ref{discrete_EEB_diagram} and \ref{continuous_EEB_diagram}. All mathematical results are the authors' original work.

\bibliographystyle{acm}
\bibliography{bibliography}	

@article {Cho75,
	AUTHOR = {Choi, Man Duen},
	TITLE = {\href{https://doi.org/10.1016/0024-3795(75)90075-0}{Completely positive linear maps on complex matrices}},
	JOURNAL = {Linear Algebra Appl.},
	VOLUME = {10},
	YEAR = {1975},
	NUMBER = {3},
	PAGES = {285--290},
}

@article {TroJA05,
	AUTHOR = {Tropp, Joel A},
	TITLE ={\href{{https://users.cms.caltech.edu/~jtropp/conf/Tro05-Complex-Equiangular-SPIE-preprint.pdf}}{Complex Equiangular Tight Frames}},
	JOURNAL = {Wavelets XI (Proceedings of SPIE)},
	VOLUME = {5914},
	YEAR = {2005},
	PAGES = {10 pp},
}

@article{FicM:MayBR21,
  author = {Fickus, Matthew and Mayo, Benjamin R.},
  title = {\href{https://ieeexplore.ieee.org/document/9281017}{Mutually unbiased equiangular tight frames}},
  journal = {IEEE Trans. Inf. Theory},
  volume = {67},
  number = {3},
  pages = {1656--1667},
  year = {2021},
}

@article {Jam72,
	AUTHOR = {Jamio{\l}kowski, A.},
	TITLE = {\href{https://doi.org/10.1016/0034-4877(72)90011-0}{Linear transformations which preserve trace and positive
	semidefiniteness of operators}},
	JOURNAL = {Rep. Math. Phys.},
	VOLUME = {3},
	YEAR = {1972},
	NUMBER = {4},
	PAGES = {275--278},
}

@article {Kra71,
	AUTHOR = {Kraus, K.},
	TITLE = {\href{https://doi.org/10.1016/0003-4916(71)90108-4}{General state changes in quantum theory}},
	JOURNAL = {Ann. Phys.},
	VOLUME = {64},
	YEAR = {1971},
	PAGES = {311--335},
	ISSN = {0003-4916},
}

@article {BD,
	AUTHOR = {Bhat, B. V. Rajarama and Devendra, Repana},
	TITLE = {\href{https://doi.org/10.48550/arXiv.2512.23598}{On regions of mixed unitarity for semigroups of unital quantum channels}},
	PAGES = {Preprint (2025), https://doi.org/10.48550/arXiv.2512.23598},
}

@article {RP,
	AUTHOR = {Pereira, Rajesh},
	TITLE = {\href{https://doi.org/10.1090/S0002-9939-05-08031-7}{Representing conditional expectations as elementary operators}},
	JOURNAL = {Proc. Amer. Math. Soc.},
	VOLUME = {134},
	YEAR = {2006},
	NUMBER = {1},
	PAGES = {253--258},
}

@article {HU,
	AUTHOR = {Umegaki, Hisaharu},
	TITLE = {\href{https://doi.org/10.2748/tmj/1178245177}{Conditional expectation in an operator algebra}},
	JOURNAL = {Tohoku Math. J. (2)},
	VOLUME = {6},
	YEAR = {1954},
	PAGES = {177--181},
}

@article {H,
	AUTHOR = {Holevo, A. S.},
	TITLE = {\href{https://doi.org/10.1070/rm1998v053n06ABEH000091}{Quantum coding theorems}},
	JOURNAL = {Russian Math. Surveys},
	VOLUME = {53},
	YEAR = {1998},
	NUMBER = {6},
	PAGES = {1295--1331},
}

@article {PPPR,
	AUTHOR = {Pandey, Satish K. and Paulsen, Vern I. and Prakash, Jitendra
	and Rahaman, Mizanur},
	TITLE = {\href{https://doi.org/10.1063/1.5045184}{Entanglement breaking rank and the existence of {SIC} {POVM}s}},
	JOURNAL = {J. Math. Phys.},
	VOLUME = {61},
	YEAR = {2020},
	NUMBER = {4},
	PAGES = {14pp},
}

@article {HHH,
	AUTHOR = {Horodecki, Micha\l{} and Horodecki, Pawe\l{} and Horodecki,
	Ryszard},
	TITLE ={\href{https://doi.org/10.1016/S0375-9601(96)00706-2}{Separability of mixed states: necessary and sufficient
	conditions}},
	JOURNAL = {Phys. Lett. A},
	VOLUME = {223},
	YEAR = {1996},
	NUMBER = {1-2},
	PAGES = {1--8},
}

@article {HO,
	AUTHOR = {Horodecki, Pawe\l},
	TITLE = {\href{https://doi.org/10.1016/S0375-9601(97)00416-7}{Separability criterion and inseparable mixed states with
	positive partial transposition}},
	JOURNAL = {Phys. Lett. A},
	VOLUME = {232},
	YEAR = {1997},
	NUMBER = {5},
	PAGES = {333--339},
}

@article {BCHW,
	AUTHOR = {B{\"a}uml, Stefan and Christandl, Matthias and Horodecki, Karol and Winter, Andreas},
	TITLE = {\href{https://doi.org/10.1038/ncomms7908}{Limitations on quantum key repeaters}},
	JOURNAL = {Nat Commun},
	VOLUME = {6},
	YEAR = {2015},
	NUMBER = {6908},
	PAGES = {5pp},
}

@article {RJP,
	AUTHOR = {Rahaman, Mizanur and Jaques, Samuel and Paulsen, Vern I.},
	TITLE = {\href{https://doi.org/10.1063/1.5024385}{Eventually entanglement breaking maps}},
	JOURNAL = {J. Math. Phys.},
	VOLUME = {59},
	YEAR = {2018},
	NUMBER = {6},
	PAGES = {11pp},
}

@article {HH,
	AUTHOR = {Horodecki, Micha\l{} and Horodecki, Pawe\l{}},
	TITLE ={\href{https://doi.org/10.1103/PhysRevA.59.4206}{Reduction criterion of separability and limits for a class of distillation protocols}},
	JOURNAL = {Phys. Rev. A},
	VOLUME = {59},
	YEAR = {1999},
	NUMBER = {6},
	PAGES = {4206--4216},
}

@article {GB,
	AUTHOR = {Gurvits, Leonid and Barnum, Howard},
	TITLE ={\href{https://doi.org/10.1103/PhysRevA.66.062311}{Largest separable balls around the maximally mixed bipartite quantum state}},
	JOURNAL = {Phys. Rev. A},
	VOLUME = {66},
	YEAR = {2002},
	PAGES = {062311},
}

@article {JW2,
	AUTHOR = {Watrous, John},
	TITLE = {\href{https://www.rintonpress.com/xxqic9/qic-9-56/0406-0413.pdf}{Mixing doubly stochastic quantum channels with the completely
	depolarizing channel}},
	JOURNAL = {Quantum Inf. Comput.},
	VOLUME = {9},
	YEAR = {2009},
	NUMBER = {5-6},
	PAGES = {406--413},
}

@article {KLPR,
	AUTHOR = {Kribs, David W. and Levick, Jeremy and Pereira, Rajesh and
	Rahaman, Mizanur},
	TITLE = {\href{https://doi.org/10.1088/1751-8121/ad2cb0}{Operator algebra generalization of a theorem of {W}atrous and
	mixed unitary quantum channels}},
	JOURNAL = {J. Phys. A},
	VOLUME = {57},
	YEAR = {2024},
	NUMBER = {11},
	PAGES = {Paper No. 115303, 25pp},
}

@article {R,
	AUTHOR = {Ruskai, Mary Beth},
	TITLE = {\href{https://doi.org/10.1142/S0129055X03001710}{Qubit entanglement breaking channels}},
	JOURNAL = {Rev. Math. Phys.},
	VOLUME = {15},
	YEAR = {2003},
	NUMBER = {6},
	PAGES = {643--662},
}

@article {HSR,
	AUTHOR = {Horodecki, Michael and Shor, Peter W. and Ruskai, Mary Beth},
	TITLE = {\href{https://doi.org/10.1142/S0129055X03001709}{Entanglement breaking channels}},
	JOURNAL = {Rev. Math. Phys.},
	VOLUME = {15},
	YEAR = {2003},
	NUMBER = {6},
	PAGES = {629--641},
}

@article {CK,
	AUTHOR = {Casazza, Peter G. and Kova{\v{c}}evi{\'c}, Jelena},
	TITLE = {\href{https://doi.org/10.1023/A:1021349819855}{Equal-norm tight frames with erasures}},
	JOURNAL = {Adv. Comput. Math.},
	VOLUME = {18},
	YEAR = {2003},
	NUMBER = {2-4},
	PAGES = {387--430},
}

@article {SH,
	AUTHOR = {Strohmer, Thomas and Heath, Jr., Robert W.},
	TITLE = {\href{https://doi.org/10.1016/S1063-5203(03)00023-X}{Grassmannian frames with applications to coding and
	communication}},
	JOURNAL = {Appl. Comput. Harmon. Anal.},
	VOLUME = {14},
	YEAR = {2003},
	NUMBER = {3},
	PAGES = {257--275},
}

@article {STDH,
	AUTHOR = {Sustik, M\'aty\'as A. and Tropp, Joel A. and Dhillon, Inderjit
	S. and Heath, Jr., Robert W.},
	TITLE = {\href{https://doi.org/10.1016/j.laa.2007.05.043}{On the existence of equiangular tight frames}},
	JOURNAL = {Linear Algebra Appl.},
	VOLUME = {426},
	YEAR = {2007},
	NUMBER = {2-3},
	PAGES = {619--635},
}

@article {S,
	AUTHOR = {Scott, A. J. },
	TITLE = {\href{https://doi.org/10.48550/arXiv.1703.03993}{SICs: Extending the list of solutions}},
	PAGES = {Preprint (2017)},
}

@article {W,
	AUTHOR = {Welch, L. },
	TITLE = {\href{https://doi.org/10.1109/TIT.1974.1055219}{Lower bounds on the maximum cross correlation of signals}},
	JOURNAL = {IEEE Trans. Inform. Theory},
	VOLUME = {20},
	YEAR = {1974},
	NUMBER = {3},
	PAGES = {397 - 399},
}

@article {XZG,
	AUTHOR = {Xia, Pengfei and Zhou, Shengli and Giannakis, Georgios B.},
	TITLE = {\href{https://doi.org/10.1109/TIT.2005.846411}{Achieving the {W}elch bound with difference sets}},
	JOURNAL = {IEEE Trans. Inform. Theory},
	VOLUME = {51},
	YEAR = {2005},
	NUMBER = {5},
	PAGES = {1900--1907},
}

@article {DK,
	AUTHOR = {Kribs, David W.},
	TITLE = {\href{https://doi.org/10.1017/S0013091501000980}{Quantum channels, wavelets, dilations and representations of $O_n$}},
	JOURNAL = {Proc. Edinb. Math. Soc. (2)},
	VOLUME = {46},
	YEAR = {2003},
	NUMBER = {2},
	PAGES = {421--433},
}

@article {EvaDE77,
    AUTHOR = {Evans, D. E.},
     TITLE = {\href{http://projecteuclid.org/euclid.cmp/1103900873}{Irreducible quantum dynamical semigroups}},
   JOURNAL = {Comm. Math. Phys.},
     VOLUME = {54},
      YEAR = {1977},
    NUMBER = {3},
     PAGES = {293--297},
}

@article {EvaDE:HoeR78,
    AUTHOR = {Evans, D. E. and H{\o}egh-Krohn, R.},
     TITLE = {\href{https://doi.org/10.1112/jlms/s2-17.2.345}{Spectral properties of positive maps on {$C\sp*$}-algebras}},
   JOURNAL = {J. London Math. Soc.},
    VOLUME = {17},
      YEAR = {1978},
    NUMBER = {2},
     PAGES = {345--355},
}

@article {FriA77,
    AUTHOR = {Frigerio, A.},
     TITLE = {\href{https://doi.org/10.1007/BF00398571}{Quantum dynamical semigroups and approach to equilibrium}},
   JOURNAL = {Lett. Math. Phys.},
    VOLUME = {2},
      YEAR = {1977},
    NUMBER = {2},
     PAGES = {79--87},
}

@article {BKT,
	AUTHOR = {Bhat, B. V. Rajarama and Kar, Samir and Talwar, Bharat},
	TITLE = {\href{https://doi.org/10.1016/j.laa.2023.08.020}{Peripherally automorphic unital completely positive maps}},
	JOURNAL = {Linear Algebra Appl.},
	VOLUME = {678},
	YEAR = {2023},
	PAGES = {191--205},
}

@article {BKT2,
	AUTHOR = {Bhat, B. V. Rajarama and Kar, Samir and Talwar, Bharat},
	TITLE = {\href{https://doi.org/10.1007/s11856-025-2830-2}{Peripheral Poisson boundary}},
	JOURNAL = {Isr. J. Math.},
	VOLUME = {273},
	YEAR = {2026},
	PAGES = {233–256},
}

@article {MR,
	AUTHOR = {Rahaman, Mizanur},
	TITLE = {\href{https://doi.org/10.1088/1751-8121/aa7b57}{Multiplicative properties of quantum channels}},
	JOURNAL = {J. Phys. A},
	VOLUME = {50},
	YEAR = {2017},
	NUMBER = {34},
	PAGES = {345302, 26},
}

@article {GKS,
	AUTHOR = {Gorini, Vittorio and Kossakowski, Andrzej and Sudarshan, E. C.
	G.},
	TITLE = {\href{https://doi.org/10.1063/1.522979}{Completely positive dynamical semigroups of {$N$}-level
	systems}},
	JOURNAL = {J. Math. Phys.},
	VOLUME = {17},
	YEAR = {1976},
	NUMBER = {5},
	PAGES = {821--825},
}

@article {L,
	AUTHOR = {Lindblad, G.},
	TITLE = {\href{http://projecteuclid.org/euclid.cmp/1103899849}{On the generators of quantum dynamical semigroups}},
	JOURNAL = {Comm. Math. Phys.},
	VOLUME = {48},
	YEAR = {1976},
	NUMBER = {2},
	PAGES = {119--130},
}

@article {SPWC,
	AUTHOR = {Sanz, Mikel and P\'erez-Garc\'ia, David and Wolf, Michael M.
	and Cirac, Juan I.},
	TITLE = {\href{https://doi.org/10.1109/TIT.2010.2054552}{A quantum version of {W}ielandt's inequality}},
	JOURNAL = {IEEE Trans. Inform. Theory},
	VOLUME = {56},
	YEAR = {2010},
	NUMBER = {9},
	PAGES = {4668--4673}
}

@article {RD,
	AUTHOR = {Rouz\'e, Cambyse and Datta, Nilanjana},
	TITLE = {\href{https://doi.org/10.1063/1.5023210}{Concentration of quantum states from quantum functional and
	transportation cost inequalities}},
	JOURNAL = {J. Math. Phys.},
	VOLUME = {60},
	YEAR = {2019},
	NUMBER = {1},
	PAGES = {012202, 22pp},
}

@article {BCGPY,
	AUTHOR = {Burgarth, D. and Chiribella, G. and Giovannetti, V. and
	Perinotti, P. and Yuasa, K.},
	TITLE = {\href{https://doi.org/10.1088/1367-2630/15/7/073045}{Ergodic and mixing quantum channels in finite dimensions}},
	JOURNAL = {New J. Phys.},
	VOLUME = {15},
	YEAR = {2013},
	PAGES = {073045, 33pp},
}

@book {Wo,
	AUTHOR = {Wolf, M. M.},
	TITLE = {\href{https://mediatum.ub.tum.de/node?id=1701036}{Quantum Channels \& Operations: A Guided Tour}},
	PUBLISHER = {Lecture Notes, TUM},
	YEAR = {2012},
}

@article {HRS,
	AUTHOR = {Hanson, Eric P. and Rouz{\'e}, Cambyse and Stilck Fran{\c{c}}a,
	Daniel},
	TITLE = {\href{https://doi.org/10.1007/s00023-020-00906-4}{Eventually entanglement breaking {M}arkovian dynamics:
	structure and characteristic times}},
	JOURNAL = {Ann. Henri Poincar{\'e}},
	VOLUME = {21},
	YEAR = {2020},
	NUMBER = {5},
	PAGES = {1517--1571},
}

@article {LG,
	AUTHOR = {Lami, L. and Giovannetti, V.},
	TITLE = {\href{https://doi.org/10.1063/1.4942495}{Entanglement-saving channels}},
	JOURNAL = {J. Math. Phys.},
	VOLUME = {57},
	YEAR = {2016},
	NUMBER = {3},
	PAGES = {Paper No. 032201, 34pp},
}

@article{RusMB:JunM:DavK:HayP:AndW12,
  title={\href{https://www.birs.ca/workshops/2012/12w5084/report12w5084.pdf}{Operator structures in quantum information theory}},
  author={Ruskai, Mary Beth and Junge, Marius and Kribs, David and Hayden, Patrick and Winter, Andreas},
  journal={Final Report, Banff International Research Station},
  year={2012}
}

@book {RocRT70,
    AUTHOR = {Rockafellar, R. Tyrrell},
     TITLE = {Convex analysis},
    SERIES = {Princeton Mathematical Series},
    VOLUME = {No. 28},
 PUBLISHER = {Princeton University Press, Princeton, NJ},
      YEAR = {1970},
}

@article {BA,
	AUTHOR = {Bhat, B. V. Rajarama and Dias, Astrid Swizell},
	TITLE = {\href{https://doi.org/10.1007/s12044-026-00877-2}{Peripheral Poisson boundary: extensions and examples}},
	JOURNAL = {Proc. Indian Acad. Sci. (Math. Sci.)},
	VOLUME = {136},
	YEAR = {2026},
	PAGES = {Article ID 0028, 21pp},
}

\end{document}